\documentclass[12pt]{amsart}

\usepackage[T1]{fontenc}
\usepackage{amsmath,amssymb}
\usepackage{amscd}
\usepackage{pb-diagram}
\newtheorem{prop}{Theorem}[section]
\newtheorem{lemma}[prop]{Lemma}
\newtheorem{corollary}[prop]{Corollary}
\newtheorem{definition}[prop]{Definition}

\newcommand{\clc}{\cdot\ldots\cdot}
\newcommand{\olo}{\otimes\ldots\otimes}
\newcommand{\plp}{+ \ldots +}
\newcommand{\wlw}{\wedge\ldots\wedge}
\newcommand{\glg}{\hat\otimes \ldots \hat\otimes}
\newcommand{\nn}{\mathbb{N}}
\newcommand{\zz}{\mathbb{Z}}
\newcommand{\QQ}{\mathbb{Q}}
\newcommand{\rr}{\mathbb{R}}
\newcommand{\cc}{\mathbb{C}}
\newcommand{\C}[1]{\mathcal{#1}}

\newcommand{\T}[1]{\textrm{#1}}
\newcommand{\E}[1]{\emph{#1}}

\newcommand{\F}[1]{\mathbf{#1}}
\newcommand{\fork}[2]{\left\{ \begin{array}{#1} #2 \end{array} \right.} 
\newcommand{\sep}[1]{& \T{#1 } &}
\newcommand{\arr}[2]{\begin{array}{#1} #2 \end{array}}
\newcommand{\mat}[2]{\left(\begin{array}{#1} #2 \end{array} \right)}
\newcommand{\q}{\qquad}
\newcommand{\qq}{\qquad \qquad}
\newcommand{\qqq}{\qquad \qquad \qquad}
\newcommand{\qqqq}{\qquad \qquad \qquad \qquad}

\newcommand{\diff}[2]{\frac{\partial{#1}}{\partial{#2}}}
\newcommand{\inner}[1]{\langle #1 \rangle}

\newcommand{\grad}{\hat\otimes}

\newcommand{\epsi}{\varepsilon}
\newcommand{\ovar}{\otimes_\varphi}
\newcommand{\govar}{\hat\otimes_\varphi}

\newcommand{\ccyc}[2]{C^{\lambda}_{#2}(#1)}
\newcommand{\ccocyc}[2]{C^{#2}_{\lambda}(#1)} 
\newcommand{\hcyc}[2]{HC_{#2}(#1)}

\newcommand{\clie}[2]{\Lambda_{#2}{#1}}
\newcommand{\hlie}[2]{H^{\T{Lie}}_{#2}(#1)}

\newcommand{\relkt}[2]{K^{\T{rel}}_{#2}(#1)}

\begin{document}
 \title{A calculation of the multiplicative character} 

\author{J. Kaad}
\thanks{email:
    \texttt{kaad@math.ku.dk}}
\maketitle
\vspace{-10pt}
\centerline{ Department of Mathematical Sciences, University of Copenhagen}
\centerline{ Universitetsparken 5, DK-2100 Copenhagen, Denmark}

\bigskip
\vspace{30pt} 

\centerline{\textbf{Abstract}} 
We give a formula, in terms of products of commutators, for the application of
the odd multiplicative character to higher Loday symbols. On our way we
construct a product on the relative $K$-groups and investigate the
multiplicative properties of the relative Chern character. 

\newpage
\tableofcontents
\newpage
\section{Introduction}
To each finitely summable Fredholm module $(F,H)$ over a $\cc$-algebra $A$,
A. Connes and M. Karoubi associate a multiplicative character on algebraic
$K$-theory 
\[
M_F : K_n(A) \to \cc / (2 \pi i)^{\lceil \frac{n}{2} \rceil } \zz
\]
The construction uses the relative $K$-groups of a unital Banach algebra and
the relative Chern character with values in continuous cyclic homology. It can
be understood as a pairing between the abelian group generated by finitely
summable Fredholm modules and algebraic $K$-theory, \cite{ConKar}. 

In the case where $n=1$ the multiplicative character has a direct
interpretation as a Fredholm determinant, see \cite{ConKar}. Likewise, in the
case where $n=2$ the multiplicative character coincides with the determinant
invariant as defined in \cite{LBrown, LBrown2}. See also \cite{Rosenberg}. In
particular, when the $\cc$-algebra $A$ is commutative we have the explicit
formula  
\begin{equation}\label{eq:twohh} 
M_F([e^a] * [e^b]) = -\T{Tr}[PaP,PbP] \in \cc/(2\pi i) \zz
\end{equation}
for the application of the multiplicative character to the Loday product
$[e^a]*[e^b] \in K_2(A)$. Furthermore this implies the independence of the
character under trace class perturbations, \cite{Kaad}. Note that a
description of the multiplicative character in terms of a (different) central
extension has also been obtained in \cite{ConKar}. We could thus try to think
of the multiplicative character as an extension of the determinant invariant
to higher $K$-theory. The aim of the present paper is then to find an analogue
of the formula \eqref{eq:twohh} in higher dimensions. This pursuit could be
justified by the large amount of research which focus on the quantity
$\T{Tr}[PaP,PbP] \in \cc$, see \cite{BergShaw, BergShaw2, CarPin3,
  CarPin4,HH}, among others. We would also like to mention the use of the
determinant invariant in relation with the Szegö limit theorem, \cite{CarPin,
  CarPin2}. 

Let us fix an odd $2p$-summable Fredholm module $(F,H)$ over a
\emph{commutative} Banach algebra. The interior Loday product makes the direct
sum of algebraic $K$-groups $\oplus_{n=1}^\infty K_n(A)$ into a graded
commutative ring, see \cite{Loday}. We can thus consider the application of
the multiplicative character to the Loday product $[e^{a_0}]* \ldots *
[e^{a_{2p-1}}] \in K_{2p}(A)$. Here $a_0,\ldots,a_{2p-1} \in M_\infty(A)$. The
main result of the present paper is then the concrete formula 
\begin{equation}\label{eq:mulexp}
\begin{split}
& M_F([e^{a_0}]* \ldots * [e^{a_{2p-1}}]) \\
& \q = (-1)^p c_p \sum_{s \in SE_{2p-1}}\T{sgn}(s) 
\T{Tr}([P\T{TR}(a_0)P, P\T{TR}(a_{s(1)})P] \clc \\
& \qqqq \q [P\T{TR}(a_{s(2p-2)})P, P\T{TR}(a_{s(2p-1)})P]) \in \cc/(2\pi i)^p \zz 
\end{split}
\end{equation}
Here $c_p \in \QQ$ is a constant and $SE_{2p-1} \subseteq \Sigma_{2p-1}$ is
the subset of permutations satisfying $s(2i) < s(2i+1)$. The operator $P =
(F+1)/2$ is the projection associated with the Fredholm module $(F,H)$. This
shows that the multiplicative character is calculizable on the subgroup of
$K_{2p}(A)$ generated by Loday products of elements in the connected component
of the identity. Note that the commutativity assumption serves to ensure the
existence of the \emph{interior} Loday product which is needed for the
calculation to make sense. 

Now, for each $2p$-tuple $(a_0,\ldots,a_{2p-1}) \in M_\infty(A)$ we define the
complex number 
\[
\begin{split}
\inner{a_0,\ldots,a_{2p-1}} 
& = (-1)^p c_p \sum_{s \in SE_{2p-1}}\T{sgn}(s) 
\T{Tr}([P\T{TR}(a_0)P, P\T{TR}(a_{s(1)})P] \clc \\ 
& \qqq \q [P\T{TR}(a_{s(2p-2)})P, P\T{TR}(a_{s(2p-1)})P]) \in \cc
\end{split}
\]
Let us then reflect a bit on what we have obtained. First of all, choosing a
different "logarithm" for $e^{a_i} \in GL_0(A)$, that is some $b_i\in
M_\infty(A)$ with $e^{b_i}=e^{a_i}$, we get that the difference 
\begin{equation}\label{eq:diffformu}
\inner{a_0,\ldots,a_i,\ldots,a_{2p-1}} -
\inner{a_0,\ldots,b_i,\ldots,a_{2p-1}} \in (2\pi i)^p\zz
\end{equation}
is in the additive group $(2\pi i)^p \zz$. That is, the quantity
\eqref{eq:diffformu} is essentially the index of some Fredholm operator. This
is an immediate Corollary of the formula \eqref{eq:mulexp}. Furthermore we get
a couple of desirable properties straight from the corresponding properties of
the Loday product and the algebraic $K$-groups, \cite{Loday}. For example, the
map 
\[
\arr{ccc}{
GL(A)^{2p} \to \cc/(2 \pi i)^p \zz 
& \q & 
(g_0,\ldots,g_{2p-1}) \mapsto M_F([g_0]* \ldots * [g_{2p-1}])
}
\] 
is multilinear and it sends each tuple with an elementary entrance, $g_i \in
E(A)$, to zero. 

The similiarity of the trace formula in \eqref{eq:mulexp} with the expression
in the bivariant case, makes us expect the following generalizations : First
of all the quantity $\inner{a_0,\ldots,a_{2p-1}} \in \cc$ should be invariant
under perturbations of the operators $Pa_i P$ by elements in the Schatten
ideal $\C L^p(H)$. Furthermore, under suitable conditions, we expect our form
to be expressible by means of an integral over the joint essential spectrum of
the operators in question. 

Finally, we would like to explain briefly how the main result is obtained. In
order to calculate the multiplicative character of some element $[x] \in
K_{2p}(A)$ the first obstacle is to construct a lift in relative $K$-theory, 
\[
\arr{ccc}{
[\gamma] \in K_{2p}^{\T{rel}}(A) & \q & \theta[\gamma] = [x]
}
\]
In the special case where the element $[x] \in K_{2p}(A)$ is given by a
product of contractible invertible operators this is accomplished by the
construction of an explicit product on the relative $K$-groups. The product
makes the direct sum of relative $K$-groups $\oplus_{n=1}^\infty
K_n^{\T{rel}}(A)$ into a graded commutative ring and the map $\theta :
\oplus_{n=1}^\infty K_n^{\T{rel}}(A) \to \oplus_{n=1}^\infty K_n(A)$ becomes a
homomorphism of graded rings (recall that $A$ is assumed to be
commutative). The question of finding the lift $[\gamma] \in
K_{2p}^{\T{rel}}(A)$ then reduces to lifting each of the elements $[g_i] \in
K_1(A)$. This is possible by the contractibility assumption. The construction
of the product is carried out in Section \ref{extrel}. On our way we also
express the second relative $K$-group as the second homology group of a
certain simplicial set. 

Having found the lift $[\gamma] \in K_{2p}^{\T{rel}}(A)$ the next problem is
to calculate the relative Chern character of the lift 
\[
\arr{ccc}{
\T{ch}^{\T{rel}} : K_{2p}^{\T{rel}}(A) \to HC_{2p-1}(A) & \q &
\T{ch}^{\T{rel}}[\gamma] = ? 
}
\]
Following the same vein of ideas we show in Section \ref{mulrel} that the
relative Chern character is a homomorphism of graded rings. This should be
understood in the following sense : The relative Chern character has degree
minus one, so the corresponding product in continuous cyclic homology has
degree plus one, 
\[
\arr{ccc}{
* : HC_{n-1}(A) \otimes_\cc HC_{m-1}(A) \to HC_{n+m-1}(A) 
& \q & x * y = x \times (sN)(y) 
}
\]
See also \cite{Loday2}. The calculation in question thus reduces to the case
of $\T{ch}^{\T{rel}} : K_1^{\T{rel}}(A) \to HC_0(A)$. The elements in
$K_1^{\T{rel}}(A)$ are represented by smooth maps $\sigma : [0,1] \to GL(A)$
mapping $0$ to the identity $1 \in GL(A)$ and the relative Chern character
essentially determines the corresponding logarithm of the endpoint
$\sigma(1) \in GL_0(A)$. 

The desired formula \eqref{eq:mulexp} can now be obtained from combinatorial
considerations on the index cocycle associated with Fredholm modules over
commutative algebras. This is carried out in Section \ref{mulcombcalc} where
the main Theorem is presented. 

We begin by giving an account of the various product structures which will be
used throughout the paper. 

{\bf Acknowledgements:} I would like to thank Ryszard Nest for his continuous
support and many helpful comments. I would also like to thank Jerome Kaminker
for the nice talk we had at the U.C. Davis on the subject of the
paper. Finally I am very grateful to Max Karoubi for giving me some valuable
indications related to his fruitful geometric viewpoint.

\section{A preliminary on various product structures in homology} 

%%%%%%%%%%%%%%%%%%%%%%%%%%%%%%%%%%%%%%%%%%8<%%%%%%%%%%%%%%%%%%%%%%%%%%%%%%%%%%%%%%

%%%%%%%%%%%%%%%%%%%%%%%%%%%%%%%%%%%%%%%%%%8<%%%%%%%%%%%%%%%%%%%%%%%%%%%%%%%%%%%%%%

\subsection{The exterior shuffle product}\label{extshuff}
Let $A$ and $B$ be unital Banach algebras. We let $A \grad B$ denote the
projective tensor product of $A$ and $B$ in the sense of Grothendieck,
\cite{Groth}. The definition of the simplicial sets $R_p(A)$ can be found in
Section \ref{extrel}. 

For each $p,q\in \nn$ we fix an isomorphism $\varphi : A^p \otimes_\zz B^q
\to (A \otimes_\zz B)^{pq}$ of $(A \otimes_\zz B)$-bimodules. We then have the
associated group homomorphisms  
\[
\begin{split}
& \ovar : GL_p(A) \times GL_q(B) \to GL_{pq}(A \otimes_\zz B) 
\qqq \T{and} \\  
& \govar = \iota \circ \ovar : GL_p(A) \times GL_q(B) \to GL_{pq}(A \grad B) 
\end{split}
\]
Here $\iota : GL_{pq}(A \otimes_\zz B) \to GL_{pq}(A \grad B)$ is induced by the
"identity" homomorphism $\iota : A \otimes_\zz B \to A \grad B$.

A pointwise version of the completed tensor product yields a map of simplicial
sets 
\[
\arr{ccc}{
\govar : R_p(A) \times R_q(B) \to R_{pq}(A \grad B) 
& \q & 
(\sigma,\tau) \mapsto (t \mapsto \sigma(t) \govar \tau(t) )
}
\]
Composition with the shuffle map \cite{May} 
\[
\T{sh} : C_*(R_p(A)) \otimes C_*(R_q(B)) \to C_*(R_p(A) \times R_q(B))
\]
therefore equips us with a chain map 
\[
\times_\varphi = \govar \circ \T{sh} : 
C_*(R_p(A)) \otimes C_*(R_q(B)) \to C_*(R_{pq}(A \grad B)) 
\]
We will refer to the sum of smooth maps 
\[
\sigma \times_\varphi \tau 
= \sum_{(\mu,\nu) \in \Sigma_{n, m}} \T{sgn}(\mu,\nu) 
s_{\nu(m-1)}\ldots s_{\nu(0)}(\sigma) 
\govar s_{\mu(n-1)}\ldots s_{\mu(0)}(\tau)
\]
as the \emph{exterior shuffle product} of $\sigma \in R_p(A)_n$ and $\tau \in
R_q(B)_m$. Here $\Sigma_{(n, m)} \subseteq \Sigma_{n+m}$ denotes the set of
$(n,m)$-shuffles. Note that it follows by Lemma \ref{conjid} that the induced
map on homology 
\[
\times : H_n(R_p(A)) \otimes H_m(R_q(B)) \to H_{n+m}(R_{pq}(A \grad B))
\]
is independent of the choice of isomorphism $\varphi : A^p \otimes_\zz B^q \to
(A \otimes_\zz B)^{pq}$.  

%%%%%%%%%%%%%%%%%%%%%%%%%%%%%%%%%%%%%%%%%%8<%%%%%%%%%%%%%%%%%%%%%%%%%%%%%%%%%%%%%%

%%%%%%%%%%%%%%%%%%%%%%%%%%%%%%%%%%%%%%%%%%8<%%%%%%%%%%%%%%%%%%%%%%%%%%%%%%%%%%%%%%

\subsection{The exterior wedge product in Lie algebra homology}\label{extlie} 

%%%%%%%%%%%%%%%%%%%%%%%%%%%%%%%%%%%%%%%%%%8<%%%%%%%%%%%%%%%%%%%%%%%%%%%%%%%%%%%%%%

%%%%%%%%%%%%%%%%%%%%%%%%%%%%%%%%%%%%%%%%%%8<%%%%%%%%%%%%%%%%%%%%%%%%%%%%%%%%%%%%%%

Let $A$ and $B$ be unital Banach algebras. For each $n \in \nn$ we let
$\clie{A}{n}$ denote the kernel of the map 
\[
\arr{ccc}{
S : \underbrace{A \glg A}_n \to \underbrace{A \glg A}_n & \q & 
S(a_1 \olo a_n) 
= \sum_{\sigma \in \Sigma_n}\T{sgn}(\sigma) a_{\sigma^{-1}(1)} \olo
a_{\sigma^{-1}(n)}  
}
\]
Remark that the Banach space $\clie{A}{n}$ identifies with the quotient of
$A^{\grad n}$ by the usual action of the symmetric group. 

By the \emph{continuous} Lie algebra homology of the (unital) Banach algebra
$A$ we will then understand the homology of the chain complex
$(\clie{A}{*},\delta)$. Here $\delta : \clie{A}{n} \to \clie{A}{n-1}$ is the
Chevalley-Eilenberg boundary map.  

We let
\[
\arr{ccc}{
\cdot \otimes 1_B : \clie{A}{*} \to \clie{(A \grad B)}{*}
\sep{and}
1_A \otimes \cdot : \clie{B}{*}  \to \clie{(A \grad B)}{*}
}
\]
denote the chain maps obtained by functoriality from the continuous algebra
homomorphisms 
\[
\arr{ccc}{
x \mapsto  x \otimes 1_B
\sep{and}
y \mapsto 1_A \otimes y  
}
\]

We then have a chain map 
\[
\wedge^E : \clie{A}{*} \otimes \clie{B}{*} \to \clie{(A \grad B)}{*}
\]
defined by 
\[
\arr{ccc}{
x \otimes y \mapsto (x \otimes 1_B) \wedge (1_A \otimes y)
& \q & x \in \clie{A}{n} \, , \, y \in \clie{B}{m}
}
\]
For each $x \in \clie{A}{n}$ and each $y \in \clie{B}{m}$ we will refer to the
element 
\[
x \wedge^E y := (x \otimes 1_B) \wedge (1_A \otimes y)
\] 
as the \emph{exterior wedge product} of $x$ and $y$. We let 
\[
\wedge^E : \hlie{A}{n} \otimes \hlie{B}{m} \to \hlie{A \grad B}{n+m}
\]
denote the induced map on continuous Lie algebra homology. The exterior wedge
product, thus defined, is seen to be associative and graded commutative on the
level of complexes. 

%%%%%%%%%%%%%%%%%%%%%%%%%%%%%%%%%%%%%%%%%%8<%%%%%%%%%%%%%%%%%%%%%%%%%%%%%%%%%%%%%%

%%%%%%%%%%%%%%%%%%%%%%%%%%%%%%%%%%%%%%%%%%8<%%%%%%%%%%%%%%%%%%%%%%%%%%%%%%%%%%%%%%

\subsection{The exterior product of degree one in cyclic
  homology}\label{extcyc} 

%%%%%%%%%%%%%%%%%%%%%%%%%%%%%%%%%%%%%%%%%%8<%%%%%%%%%%%%%%%%%%%%%%%%%%%%%%%%%%%%%%

%%%%%%%%%%%%%%%%%%%%%%%%%%%%%%%%%%%%%%%%%%8<%%%%%%%%%%%%%%%%%%%%%%%%%%%%%%%%%%%%%%

Let $A$ and $B$ be unital Banach algebras. We let 
\[
\times : \big( C(A) \otimes C(B)\big)_* \to C_*(A \grad B)
\]
denote the exterior shuffle product on the continuous Hochschild
complex, \cite[Section $4.2$]{Loday2}.  Furthermore, we let $(\ccyc{A}{*},b)$
denote the \emph{continuous} cyclic complex. Thus in each degree $n \in
\nn\cup \{0\}$ we have a Banach space $\ccyc{A}{n}$, \cite{ConKar,
  Karoubi}. Remark that the image $\T{Im}(1 -t) \subseteq A \grad A^{\grad n}$
is closed since it coincides with the kernel of the norm operator $N : A \grad
A^{\grad n} \to A \grad A^{\grad n}$. 

By the exterior product of degree one in continuous cyclic homology we will
understand the map 
\[
* : \ccyc{A}{n} \otimes \ccyc{B}{m} \to \ccyc{A \grad B}{n+m+1}
\]
defined by 
\[
\arr{ccc}{
x*y = x \times (sNy) & \q & x \in \ccyc{A}{n} \, , \, y \in \ccyc{B}{m} 
}
\]
Here $N : C_m(B) \to C_m(B)$ is the norm operator $N = 1+t \plp t^m$ and $s :
C_m(B) \to C_{m+1}(B)$ is the extra degeneracy. 

We will need to show that the product is well defined. For this, consider the
map  
\[
\arr{ccc}{
E : C_n(A) \to \clie{M_{n+1}(A)}{n+1} & \q & 
(a_0,\ldots,a_n) \mapsto E_{12}(a_0) \wlw E_{(n+1)1}(a_n)
}
\]
where $E_{ij}(a)$ denotes the elementary matrix with $a \in A$ in position
$(i,j)$ and zeros elsewhere, \cite{LodQuill, Tsygan}.

From Theorem \ref{antprod} and Theorem \ref{traceprod} we then get the
equality 
\[
\arr{ccc}{
x * y = (\T{TR} \circ \varepsilon)(E(x) \wedge^E E(y)) 
& \q & x \in C_n(A) \, , \, y \in C_m(B)
}
\]
where $\varepsilon : \clie{M_k(A)}{*} \to \ccyc{M_k(A)}{*-1}$ and $\T{TR} :
\ccyc{M_k(A)}{*} \to \ccyc{A}{*}$ denote the antisymmetrization map and the
generalized trace on continuous cyclic homology respectively. It follows that
the product is well defined and that it is associative and graded commutative
on the level of complexes. 

Lastly, the Hochschild boundary is a (shifted) graded derivation with respect
to the product  
\[
b(x*y)= (bx)*y + (-1)^{\T{deg}(x)+1} x * (by)
\]
It follows that our multiplication descends to an exterior product of degree
one on continuous cyclic homology 
\[
* : \hcyc{A}{n} \otimes_\cc \hcyc{B}{m} \to \hcyc{A \grad B}{n+m+1}
\]

In the case where the unital Banach algebra $A$ is commutative we get an
interior product 
\[
* : \hcyc{A}{n} \otimes_\cc \hcyc{A}{m} \to \hcyc{A}{n+m+1}
\]
by composition of the exterior product with the map induced by the
multiplication $\nabla_* : A \grad A \to A$. 

For further details on the constructions given in this section we refer to
\cite{Loday2,Tsygan2}. 

%%%%%%%%%%%%%%%%%%%%%%%%%%%%%%%%%%%%%%%%%%8<%%%%%%%%%%%%%%%%%%%%%%%%%%%%%%%%%%%%%%

%%%%%%%%%%%%%%%%%%%%%%%%%%%%%%%%%%%%%%%%%%8<%%%%%%%%%%%%%%%%%%%%%%%%%%%%%%%%%%%%%%

%%%%%%%%%%%%%%%%%%%%%%%%%%%%%%%%%%%%%%%%%%8<%%%%%%%%%%%%%%%%%%%%%%%%%%%%%%%%%%%%%%

%%%%%%%%%%%%%%%%%%%%%%%%%%%%%%%%%%%%%%%%%%8<%%%%%%%%%%%%%%%%%%%%%%%%%%%%%%%%%%%%%%

\section{An exterior product on the relative $K$-theory of Banach
  algebras}\label{extrel} 
Let $A$ be a unital Banach algebra. Before giving the construction of the
exterior product, we recall the definition of the relative $K$-groups, as
introduced by M. Karoubi, \cite{Karoubi}. 

To this end, for each $n \in \nn_0$ we let $\Delta^n = \{(t_1,\ldots,t_n) \in
\rr^n\, | \, t_i \geq 0 \, , \, \sum_{i=1}^n t_i \leq 1\}$ denote the standard
$n$-simplex and we let $\F{0},\ldots,\F{n} \in \Delta^n$ denote the vertices. 

For each $p \in \nn \cup \{\infty\}$ we then associate a simplicial set
$R_p(A)$. In degree $n \in \nn_0$ it is given by the set of normalized
continuous maps  
\[
\arr{ccc}{
\sigma : \Delta^n \to GL_p(A)
& \q & \sigma(\F{0})=1_p 
}
\]
The face operators and degeneracy operators are given by 
\[
\begin{split}
& d_i(\sigma)(t_1,\ldots,t_{n-1}) = 
\fork{ccc}{
\sigma(1-\sum_{j=1}^{n-1}t_j,t_1,\ldots,t_{n-1})\cdot \sigma(\F{1})^{-1} 
\sep{for} 
j = 0 \\
\sigma(t_1,\ldots,t_{i-1},0,t_i,\ldots,t_{n-1})
\sep{for} j \in \{1,\ldots,n\} \\
} \\
& s_j(\sigma)(t_1,\ldots,t_{n+1}) = 
\fork{ccc}{
\sigma(t_2,\ldots,t_{n+1}) 
\sep{for} j = 0 \\
\sigma(t_1,\ldots,t_{i-1},t_i + t_{i+1},\ldots,t_{n+1}) 
\sep{for} j \in \{1,\ldots,n\} 
}
\end{split}
\]
Remark the extra factor $\sigma(\F{1})^{-1}$ in the expression for $d_0 :
R_p(A)_n \to R_p(A)_{n-1}$. We will often refer to the simplicial set
$R_\infty(A)$ by $R(A)$. 

The fundamental group of the pointed Kan complex $R_p(A)$ is given by 
\[
\pi_1(R_p(A)) = R_p(A)_1/\sim 
\]
where $\sim$ denotes the equivalence relation of homotopies with fixed end
points. The group structure is given by pointwise multiplication. For each $p
\in \{3,4,\ldots \} \cup \{\infty\}$ the commutator subgroup is seen to be
perfect and we can thus apply the plus construction to the geometric
realization of $R_p(A)$. 

\begin{definition}\cite{Karoubi}
By the \emph{relative $K$-groups} of the unital Banach algebra $A$ we will
understand the homotopy groups of the pointed topological space $|R(A)|^+$, 
\[
K_n^{\E{rel}}(A) := \pi_n(|R(A)|^+)
\]
\end{definition}

The relative $K$-groups relate the algebraic $K$-groups and topological
$K$-groups through the long exact sequence 
\begin{equation}\label{eq:xseqk}
 \dgARROWLENGTH=0.5\dgARROWLENGTH
\begin{diagram}
\node{\ldots } \arrow{e,t}{i}
     \node{K_{n+1}^{\T{top}}(A)} \arrow{e,t}{v}
          \node{K_n^{\T{rel}}(A)} \arrow{e,t}{\theta}
               \node{K_n(A)} \arrow{s,r}{i} \\
               \node{\ldots } 
          \node{K_{n-1}(A)} \arrow{w,b}{i} 
     \node{K_{n-1}^{\T{rel}}(A)} \arrow{w,b}{\theta}     
\node{K_n^{\T{top}}(A)} \arrow{w,b}{v}
\end{diagram}
\end{equation}
Here the map $\theta : K_n^{\T{rel}}(A) \to K_n(A)$ is induced by the map of
simplicial sets 
\[
\arr{ccc}{
\theta : R_p(A) \to BGL_p(A) & \q & 
\sigma \mapsto 
(\sigma(\F{0})\sigma(\F{1})^{-1},\ldots \sigma(\F{n-1})\sigma(\F{n})^{-1}) 
}
\]

For later use we give a result on the homology of the simplicial set
$R(A)$. For this, let $R^\infty(A)$ denote the corresponding simplicial set
with continuous maps replaced by smooth maps. We then have the following
isomorphism,  

\begin{lemma}\label{consmooth}
The simplicial map $i : R^\infty(A) \to R(A)$ given by inclusion, yields an
isomorphism in homology, $H_n(R^\infty(A)) \cong H_n(R(A))$. 
\end{lemma}  
\begin{proof}
For each $r > 0$ we choose a smooth map $\delta_r \in C_c^\infty(\rr^n)$ which
is supported on the ball around $0$ with radius $r$ and which satisfies
$\delta_r \geq 0$ and $\int_{\rr^n} \delta_r dx= 1$. 

For each compactly supported continuous function $\sigma \in
C_c(\rr^n,M_p(A))$ with $\sigma(t) \in GL_p(A)$ for all $t \in \Delta^n$ we can
define the continuous homotopy $H(\sigma) : [0,1] \times \Delta^n \to M_p(A)$  
\[
H(\sigma)(r,t) = \fork{ccc}{
(\delta_r * \sigma)(t) & \q & r > 0 \\ 
\sigma(t) & \q & r = 0 
}  
\]
Here $*$ denotes the convolution product of compactly supported maps. We note
that $H(\sigma)(r, \cdot ) : \Delta^n \to GL_p(A)$ for sufficiently small $r >
0$. The result is then a consequence of the $GL_p(A)$-invariance of the
homotopy, $H(\sigma \cdot g) = H(\sigma) \cdot g$. 
\end{proof}

\subsection{A calculation of the second relative $K$-group}\label{calcre}
In this section we show that the second relative $K$-group is isomorphic to
the second homology group of a certain simplicial set. The result is thus
similar to the result in \cite{Loday, Quillen, Rosenberg}. Here the second
algebraic $K$-group of a ring (using Quillen's definition) is shown to agree
with the second homology group of the elementary matrices over the ring. The
present calculation is relevant for our explicit definition of the exterior
product on the relative $K$-groups. 

Let $F(A)_1 \subseteq R(A)_1$ denote the smallest normal subgroup of $R(A)_1$
which contains the commutators 
\[
\arr{ccc}{
\sigma \tau \sigma^{-1} \tau^{-1} \in F(A)_1 
& \q & 
\forall \sigma,\tau \in R(A)_1
}
\] 
and which satisfies 
\[
\arr{ccc}{
\alpha \in F(A)_1 \Rightarrow \overline{\alpha} \in F(A)_1 
\sep{and} 
\alpha \in F(A)_1 \Rightarrow \big(g \alpha g^{-1} \in F(A)_1 \q \forall g \in
GL(A) \big) 
}
\]
Here $\overline{\alpha} : t \mapsto \alpha(1-t)\cdot \alpha(\F{1})^{-1}$
denotes the inverse path. 

Now, we let $\sim_F$ denote the equivalence relation on $R(A)$ which is
defined degreewise by  
\[
x \sim_F y \Leftrightarrow \big( \forall I \in \Delta[n]_{n-1} \exists \tau \in
F(A)_1 : d_I(x) \tau = d_I(y) \big)
\]
Here $d_I = d_{i_1}\ldots d_{i_{n-1}} : R(A)_n \to R(A)_1$ for each $0 \leq i_1
< \ldots < i_{n-1} \leq n$. We let $Q(A)= R(A)/\sim_F$ denote the quotient
simplicial set. 

\begin{prop}
The quotient map $\pi : R(A) \to Q(A)$ is a Kan fibration. 
\end{prop}
\begin{proof}
Suppose that $x_0,\ldots,x_{k-1},x_{k+1},\ldots,x_n \in R(A)_{n-1}$ are
compatible and suppose that $y \in Q(A)_n$ satisfies $d_i(y)=\pi(x_i)$ for all
$i \neq k$. We will only consider the problematic case of $n=2$. 

We start by choosing a $z \in R(A)_2$ with $\pi(z)=y$. Suppose that $k=0$. For
each $i \in \{1,2\}$ we choose a $\tau_i \in F(A)_1$ with $d_i(z) \cdot \tau_i
= x_i$. We define the appropriate lift $x \in R(A)_2$ by the formula 
\[
\arr{ccc}{
x(t_1,t_2) = z(t_1,t_2) \tau_1(t_2)\tau_2(t_1) 
& \q & \forall (t_1,t_2) \in \Delta^2 
}
\]
A calculation of the faces of $x \in R(A)_2$ then shows that $\pi(x)=y$, that
$d_1(x)=x_1$, and that $d_2(x)=x_2$.  

The cases where $k=1$ or $k=2$ are treated in a similar fashion. The
appropriate lifts are given by 
\[
\begin{split}
x(t_1,t_2) 
& =  z(t_1,t_2) \cdot z(1)^{-1}\tau_0(t_2)z(1) \cdot \tau_2(t_1+t_2) 
\qqq \T{and} \\
x(t_1,t_2) 
& = z(t_1,t_2) \cdot z(1)^{-1}\tau_1(1-t_1)\tau_1(1)^{-1} z(1) \cdot \tau_2(t_1+t_2)
\end{split}
\]
respectively. 
\end{proof}

Let $F(A)$ denote the fiber of the Kan fibration $\pi : R(A) \to Q(A)$

\begin{lemma}\label{calchom}
The inclusion $i : F(A) \to R(A)$ induces an isomorphism 
\[
\arr{ccc}{
\pi_n(i) : \pi_n(F(A)) \to \pi_n(R(A)) 
\sep{for all} n \geq 2
}
\]
The fundamental group of the fiber $F(A)$ equals the commutator subgroup of
$\pi_1(R(A))$ and the induced map 
\[
\pi_1(i) : \big[\pi_1(R(A)),\pi_1(R(A))\big] \to \pi_1(R(A))
\]
is the inclusion. 
\end{lemma}
\begin{proof}
We will only consider the calculation of the fundamental group of the
fiber. For this, note that  
\[
\arr{ccc}{
\pi_1(F(A)) = F(A)_1 / \sim \sep{and} \pi_1(R(A)) = R(A)_1 / \sim
}
\]
where $\sim$ in both cases denotes homotopies through $GL_p(A)$ with fixed
endpoints. The proof of the lemma is then essentially a matter of checking
that each element in $\pi_1(F(A))$ can be represented by a commutator of
elements in $R(A)_1$. This follows since 
\[
\arr{ccc}{
\overline{\alpha} \sim \alpha^{-1} \sep{for each} \alpha \in R(A)_1
}
\]
and since  
\[
\arr{ccc}{
g\alpha g^{-1} \sim \gamma_g 
\mat{cc}{\alpha_{1/2} & 0 \\
0 & 1 } 
(\gamma_g)^{-1} 
\sep{for each} g \in GL(A) \, , \, \alpha \in R(A)_1
}
\]
Here $\alpha_{1/2}(t) = 1$ for all $t \in [0,1/2)$ and $\gamma_g \in R(A)_1$
satisfies $\gamma_g(t) = \mat{cc}{g & 0   \\ 0   & g^{-1}}$ for all $t \in
[1/4, 1]$. 
\end{proof}

\begin{corollary}
The homotopy groups of $Q(A)$ are given by 
\[
\pi_n(Q(A)) = \fork{ccc}{
0 
\sep{for} n \geq 2 \\
\pi_1(R(A))
/[\pi_1(R(A)),\pi_1(R(A))] 
\sep{for} n=1
}
\]
The induced map $\pi_1(\pi) : \pi_1(R(A)) \to \pi_1(Q(A))$ is the quotient
map. 
\end{corollary}
\begin{proof}
This is immediate from the long exact sequence of homotopy groups associated
with the Kan fibration $F(A) \to R(A) \to Q(A)$ and Lemma \ref{calchom}. 
\end{proof}

Let $\C F$ denote the homotopy fiber of the map $\pi^+ : |R(A)|^+ \to |Q(A)|$
induced by the quotient map $\pi : R(A) \to Q(A)$. 

\begin{prop}
The inclusion $i : F(A) \to R(A)$ gives rise to a homotopy equivalence 
\[
f^+ : |F(A)|^+ \to \C F
\]
\end{prop}
\begin{proof}
This follows from \cite{Berrick} since $\pi_1(Q(A))$ is abelian. 
\end{proof}

With this precise description of the homotopy fiber $\C F$ in hand, we are
able to obtain the desired calculation of the second relative $K$-group. 

\begin{corollary}\label{simpleisom}
The space $|F(A)|^+$ is simply connected and the map 
\[
i^+ : |F(A)|^+ \to |R(A)|^+
\]
induced by the inclusion $i : F(A) \to R(A)$ yields an isomorphism 
\[
\pi_n(i^+) : \pi_n(|F(A)|^+) 
\to \pi_n(|R(A)|^+) 
= K_n^{\E{rel}}(A) 
\] 
for all $n\geq 2$. 
\end{corollary}
\begin{proof}
This is immediate from the long exact sequence of homotopy groups arising from
the fibration $|F(A)|^+ \to |R(A)|^+ \to |Q(A)|$. 
\end{proof}

\begin{corollary}
The second relative $K$-group of the unital Banach algebra $A$ is isomorphic
to the second homology group of the simplicial set $F(A)$ 
\[
K_2^{\E{rel}}(A) \cong H_2(F(A))
\]
\end{corollary}
\begin{proof}
Since $|F(A)|^+$ is simply connected the Hurewicz homomorphism  
\[
h_2 : \pi_2(|F(A)|^+) \to H_2(|F(A)|^+) \cong H_2(|F(A)|) \cong H_2(F(A))
\]
is an isomorphism. But the group $\pi_2(|F(A)|^+)$ is isomorphic to
$K_2^{\T{rel}}(A)$ by Corollary \ref{simpleisom}. 
\end{proof}

\subsection{An H-group structure on $|R(A)|^+$}\label{hstruct}
In this section we show that a pointwise version of the direct sum on $GL(A)$
determines a commutative $H$-group structure on $|R(A)|^+$. Our exposition
will follow \cite[Section $1.2$]{Loday} and \cite{Wagoner} closely. 

Let $A$ be a unital Banach algebra. We define the direct sum on the simplicial
set $R(A)$ as a pointwise version of the direct sum on $GL(A)$, thus 
\[
\arr{ccc}{
\oplus : R(A) \times R(A) \to R(A) & \q & (\sigma,\tau) \mapsto (t \mapsto
\sigma(t) \oplus \tau(t) ) 
}
\]
Let $\oplus^+ : |R(A) \times R(A)|^+ \to |R(A)|^+$ denote the map induced by
functoriality of the geometric realization and the
plus-construction. Furthermore, let 
\[
k^{-1} : |R(A)|^+ \times |R(A)|^+ \to |R(A) \times R(A)|^+ 
\]
denote some homotopy inverse of the homotopy equivalence given by the
projection onto each factor. We then define the sum on $|R(A)|^+$ as the
composition 
\[
+ = \oplus^+ \circ k^{-1} : |R(A)|^+ \times |R(A)|^+ \to |R(A)|^+ 
\]
This will be the composition in our commutative $H$-group structure on
$|R(A)|^+$. The neutral element will be given by the constant map $1 :
\Delta^n \to GL(A)$. 

Now, to each injection $u : \nn \to \nn$ there is a group homomorphism $u :
GL(A) \to GL(A)$ defined by 
\[
u(g)_{ij} = \fork{ccc}{
g_{kl} \sep{for} i=u(k), j=u(l) \\
\delta_{ij} \sep{\q } \T{elsewhere}
}
\]
We extend this construction to a pointwise version, associating a simplicial map 
\[
\arr{ccc}{
u : R(A) \to R(A) & \q & \sigma \mapsto \big(t \mapsto u(\sigma(t))\big)
}
\]
to each injective map $u : \nn \to \nn$. We let $u^+ : |R(A)|^+ \to |R(A)|^+$
denote the map induced by functoriality. 

\begin{lemma}\label{conjid}
For each elementary matrix $g \in E(A)$ the (pre)-simplicial maps 
\[
\arr{ccc}{
\E{Ad}_g : R(A) \to R(A) \sep{and} \E{Ad}_g : F(A) \to F(A) 
}
\]
given by $\sigma \mapsto g \sigma g^{-1}$ are homotopic to the identity. In
particular, for each injection $u : \nn \to \nn$ we have 
\[
\arr{ccc}{
u_* = \E{Id} : H_*(R(A)) \to H_*(R(A)) 
\sep{and}
u_* = \E{Id} : H_*(F(A)) \to H_*(F(A))
}
\]
\end{lemma}
\begin{proof}
We will only consider the case of $F(A)$. Let $g \in E(A)$. Let $\gamma \in
F(A)_1$ satisfy $\gamma(0)=1$ and $\gamma(1)=g$. We then define a
presimplicial homotopy $h_i : F(A)_n \to F(A)_{n+1}$ by 
\[
h_i(\sigma) = (s_n \ldots s_{i+1} s_{i-1} \ldots s_0)(\gamma) \cdot
s_i(\sigma) 
\]
proving the first statement of the lemma. To prove the second statement, let
$u : \nn \to \nn$ be injective and note that to each \emph{finite} number of
elements $\sigma_1,\ldots,\sigma_m \in F(A)_n$ there is an elementary matrix
$g \in E(A)$ such that 
\[
\arr{ccc}{
u(\sigma_i) = g\cdot \sigma_i \cdot g^{-1} 
\sep{for all} i \in \{1,\ldots,m\} 
}
\]
Since each element in $x \in H_n(F(A))$ is represented by a finite number of element
in $F(A)_n$ we must have $u_*(x) = x$. 
\end{proof}

The results obtained in Section \ref{calcre} together with the classical
Whitehead theorem now allows us to show that the monoid of injections $u :
\nn \to \nn$ acts on $|R(A)|^+$ by homotopy equivalences.  

\begin{prop}
For each injection $u : \nn \to \nn$ the induced map 
\[
u^+ : |R(A)|^+ \to |R(A)|^+
\]
is a homotopy equivalence. 
\end{prop}
\begin{proof}
We show that 
\[
u^+ : |R(A)|^+ \to |R(A)|^+ 
\]
is a weak equivalence and refer to Whitehead's theorem. 

For $n = 1$ we note that $\pi_1(|R(A)|^+) \cong H_1(R(A))$, so $\pi_1(u^+) :
\pi_1(|R(A)|^+) \to \pi_1(|R(A)|^+)$ is an isomorphism by Lemma \ref{conjid}. 

For $n \geq 2$ we note that $\pi_n(|F(A)|^+) \cong \pi_n(|R(A)|^+)$ by
Corollary \ref{simpleisom}. Since the space $|F(A)|^+$ is simply connected, we
will only need to show that $u^+ : |F(A)|^+ \to |F(A)|^+$ induces an
isomorphism in homology. However, this is a consequence of Lemma \ref{conjid}. 
\end{proof}

The commutative $H$-group properties of the sum $+ : |R(A)|^+ \times |R(A)|^+
\to |R(A)|^+$ and the neutral element $1 \in |R(A)|^+$ can now be obtained by
a rephrasing of the arguments in \cite[Section $1.2$]{Loday}. 

\begin{corollary}\label{pseudoid}
For each injection $u : \nn \to \nn$ the induced map 
\[
u^+ : |R(A)|^+ \to |R(A)|^+
\]
is homotopic to the identity. 
\end{corollary}
\begin{proof}
This is a consequence of the Groethendieck group of the monoid of injections
$u : \nn \to \nn$ being trivial. See \cite[Lemma $1.2.8$]{Loday}. 
\end{proof}

\begin{prop}\label{hgroup}
The application $+ : |R(A)|^+ \times |R(A)|^+ \to |R(A)|^+$ and the neutral
element $1 \in |R(A)|^+$ define a commutative $H$-group structure on
$|R(A)|^+$. 
\end{prop}
\begin{proof}
That the sum and the neutral element defines a homotopy associative and
homotopy commutative $H$-space structure follows from Corollary \ref{pseudoid}
since the appropriate maps are homotopic up to composition with some $u^+ :
|R(A)|^+ \to |R(A)|^+$ . The existence of a homotopy inverse is automatic
since we are working exclusively with connected $CW$-complexes, \cite[Theorem
$3.4$]{Stasheff}. 
\end{proof}

We end this section by showing that the map $\theta : |R(A)|^+ \to BGL(A)^+$,
induced by the simplicial map $\theta : \sigma \mapsto
(\sigma(\F{0})\sigma(\F{1})^{-1}, \ldots, \sigma(\F{n-1})\sigma(\F{n})^{-1})$,
respects the $H$-group structures. 

\begin{prop}\label{resph}
The map $\theta : |R(A)|^+ \to BGL(A)^+$ is an $H$-map. 
\end{prop}
\begin{proof}
This is essentially a matter of checking that the simplicial maps given by 
\[
\theta \circ \oplus \, \T{ and } \,  \oplus \circ (\theta \times \theta) : 
R(A) \times R(A) \to BGL(A)
\]
coincide. 
\end{proof}

\subsection{Construction of the product in relative $K$-theory}\label{conprod} 
In this section we show that a pointwise version of the exterior Loday
product, determines a multiplicative structure on the relative $K$-groups. The
exposition will follow \cite[Section $2.1$]{Loday} closely. 

Let $A$ and $B$ be unital Banach algebras. 

Let $p,q \in \{3,4,\ldots\}$ be fixed and let $\varphi : A^p \otimes_\zz B^q
\to (A \otimes_\zz B)^{pq}$ denote some isomorphism of $A \otimes_\zz
B$-bimodules. As in Section \ref{extshuff} we have a corresponding map of
simplicial sets  
\[
\govar : R_p(A) \times R_q(B) \to R_{pq}(A \grad B) 
\]
We let $\govar^+ : |R_p(A) \times R_q(B)|^+ \to |R_{pq}(A \grad B)|$ denote
the induced map between the plus-constructions. Furthermore, let 
\[
k^{-1} : |R_p(A)|^+ \times |R_q(B)|^+ \to |R_p(A) \times R_q(B)|^+
\]
denote some homotopy inverse to the map given by the projection onto each
factor. We then define the tensor product 
\[
\grad^+ = \iota^+ \circ \govar^+ \circ k^{-1} :
|R_p(A)|^+ \times |R_q(B)|^+ \to |R(A \grad B)|^+ 
\]
Here $\iota^+ : |R_{pq}(A \grad B)|^+ \to |R(A \grad B)|^+$ is determined by
functoriality from the inclusion $\iota : R_{pq}(A \grad B) \to R(A \grad B)$. 

Following the argumentation of Theorem \ref{hgroup} we see that the tensor
product thus defined is natural in $A$ and $B$, bilinear, associative and
commutative up to homotopy. Furthermore, it only depends on the choice of
isomorphism $\varphi : A^p \otimes_\zz B^q \to (A \otimes_\zz B)^{pq}$ of $(A
\otimes_\zz B)$-bimodules up to homotopy. See also \cite[Section
$2.1.2$]{Loday}. 

Now, in order to get a map which descends to the smash product and which
behaves well when $p$ and $q$ tend to infinity we define 
\[
\begin{split}
& \gamma_{p,q}^{\T{rel}} : |R_p(A)|^+ \times |R_q(B)|^+ \to |R(A \grad B)|^+   \\
& \qqq  
\gamma_{p,q}^{\T{rel}} : (x,y) \mapsto x \grad^+ y - x \grad^+ 1_q - 1_p
\grad^+ y + 1_p \grad^+ 1_q   
\end{split}
\]
Here the minus sign comes from the (commutative) $H$-group structure on $|R(A
\grad B)|^+$ defined in Section \ref{hstruct}. The elements $1_p \in
|R_p(A)|^+$ and $1_q \in |R_q(B)|^+$ are given by the constant maps $1_p :
\Delta^n \to GL_p(A)$ and $1_q : \Delta^n \to GL_q(B)$.  

It is then immediate that the restriction $\gamma_{p,q}^{\T{rel}} : |R_p(A)|^+
\vee |R_q(B)|^+ \to |R(A \grad B)|^+$ is homotopically trivial. Since $|R(A
\grad B)|^+$ is an $H$-group we thus get a map on the smash product 
\[
\widehat{\gamma}_{p,q}^{\T{rel}} : |R_p(A)|^+ \wedge |R_q(B)|^+ \to |R(A \grad B)|^+
\]
which is unique up to homotopy and which makes the maps 
\[
\widehat{\gamma}_{p,q}^{\T{rel}} \circ \pi \, \T{ and } \,
\gamma_{p,q}^{\T{rel}} : |R_p(A)|^+ \times |R_q(B)|^+ \to |R(A \grad B)|^+ 
\]
homotopic. Here $\pi : |R_p(A)|^+ \times |R_q(B)|^+ \to |R_p(A)|^+ \wedge
|R_q(B)|^+$ denotes the quotient map.  

The argumentation explicited in \cite[p.333-335]{Loday} now ensures the
existence of a continuous map 
\[
\widehat{\gamma}^{\T{rel}} : |R(A)|^+ \wedge |R(B)|^+ \to |R(A \grad B)|^+
\]
which is natural in $A$ and $B$, bilinear, associative and commutative up to
weak homotopies. Furthermore, for any $p,q \in \{3,4,\ldots\}$ the maps
\[
\widehat{\gamma}_{p,q}^{\T{rel}} 
\, \T{ and } \, \widehat{\gamma}^{\T{rel}} \circ ( \iota \wedge \iota ) 
: |R_p(A)|^+ \wedge |R_q(B)|^+ \to |R(A \grad B)|^+ 
\]
agree up to weak homotopy. Here $\iota \wedge \iota : |R_p(A)|^+ \wedge
|R_q(B)|^+ \to |R(A)|^+ \wedge |R(B)|^+$ denotes the inclusion. This enables
us to make the following definition. 

\begin{definition} 
By the \emph{exterior product in relative $K$-theory} we understand the map 
\[
*^{\E{rel}} : K_n^{\E{rel}}(A) \times K_m^{\E{rel}}(B) 
\to K_{n+m}^{\E{rel}}(A \grad B)
\]
given by the formula 
\[
[f]*^{\E{rel}} [g] = [\widehat{\gamma}^{\E{rel}}\circ (f \wedge g)] 
\]
for each $[f] \in \pi_n(|R(A)|^+)$ and $[g] \in \pi_m(|R(B)|^+)$. 
\end{definition}

The naturality, bilinearity and associativity up to weak homotopies of the map
$\widehat{\gamma}^{\T{rel}} : |R(A)|^+ \wedge |R(B)|^+ \to |R(A \grad B)|^+$
imply the corresponding properties for the exterior product. 

\begin{prop}
The exterior product in relative $K$-theory 
\[
*^{\E{rel}} : K_n^{\E{rel}}(A) \times K_m^{\E{rel}}(B) 
\to K_{n+m}^{\E{rel}}(A \grad B)
\]
is natural, bilinear and associative.  
\end{prop}

In the case where the unital Banach algebra $A$ is commutative we get an
interior product 
\[
*^{\T{rel}} : K_n^{\T{rel}}(A) \times K_m^{\T{rel}}(A) 
\to K_{n+m}^{\T{rel}}(A)
\]
given by composition of the exterior product with the map induced by the continuous
algebra homomorphism $\nabla : A \grad A \to A \, \,$, $\, a_1 \otimes a_2 \mapsto a_1
\cdot a_2$. We are thus able to equip the direct sum of relative $K$-groups
$\oplus_{n=1}^\infty K_n^{\T{rel}}(A)$ with the structure of a graded
commutative ring. See also \cite[Theorem $2.1.12$]{Loday}. 

\subsection{Relations with the Loday product} 
Our task is now to compare the exterior product of Loday in algebraic
$K$-theory with the exterior product in relative $K$-theory. 

Let $A$ and $B$ be unital Banach algebras. Recall that the Loday product 
\[
* : K_n(A) \times K_m(B) \to K_{n+m}(A \otimes_\zz B)
\]
is uniquely determined by the continuous maps 
\[
\begin{split}
& \gamma_{p,q} : BGL_p(A)^+ \times BGL_q(B)^+ \to BGL(A \otimes_\zz B)^+  \\ 
& \qqq (x,y) \mapsto x \otimes^+ y - 1_p \otimes^+ y - x \otimes^+ 1_q 
+ 1_p \otimes^+ 1_q 
\end{split} 
\]
Here the tensorproduct $\otimes^+ : BGL_p(A)^+ \times
BGL_q(B)^+ \to BGL(A \otimes_\zz B)^+$ is induced by the group homomorphism 
\[
\ovar : GL_p(A) \times GL_q(B) \to GL_{pq}(A \otimes_\zz B) \subseteq GL(A
\otimes_\zz B) 
\]
associated with an isomorphism $\varphi : A^p \otimes_\zz B^q \to (A
\otimes_\zz B)^{pq}$ of $A \otimes_\zz B$-bimodules. The additive compositions
come from the (commutative) $H$-group structure on $BGL(A \otimes_\zz B)^+$,
\cite{Loday}. 

\begin{definition}
By the \emph{completed Loday product} in algebraic $K$-theory we will
understand the composition 
\[
\hat{*} = \iota_* \circ * : K_n(A) \times K_m(B) \to K_{n+m}(A \grad B) 
\]
of the Loday product and the map induced by the "identity" ring homomorphism
$\iota : A \otimes_\zz B \to A \grad B$ 
\end{definition}

We can then show that the homomorphism $\theta : K_n^{\T{rel}}(A) \to K_n(A)$
respects the exterior product structures. 

\begin{prop}\label{thetprod}
For each $x \in K_n^{\E{rel}}(A)$ and each $y \in K_m^{\E{rel}}(B)$ we have
the equality 
\[
\theta(x *^{\E{rel}} y) = \theta(x)\, \hat{*}\, \theta(y) 
\]
in $K_{n+m}(A \grad B)$. In particular, the map $\theta : \oplus_{n \geq 1}
K_n^{\E{rel}}(A) \to \oplus_{n \geq 1} K_n(A)$ is a homomorphism of graded
commutative rings whenever $A$ is a commutative, unital Banach algebra. 
\end{prop} 
\begin{proof} Let $p,q \in \{3,4,\ldots\}$ and let $\varphi : A^p \otimes_\zz
  B^q \to (A \otimes_\zz B)^{pq}$ denote some isomorphism of $(A \otimes_\zz
  B)$-bimodules. On the level of simplicial sets we then have the equality 
\[
\arr{ccc}{
\theta(\sigma \govar \tau) = 
\iota ( \theta(\sigma) \ovar \theta(\tau) )
\sep{for all} 
\sigma \in R_p(A)_n \, , \, \tau \in R_q(B)_n 
}
\]
This shows that the maps 
\[
\theta \circ \grad^+  
\, \T{ and } \, 
\iota^+ \circ \otimes^+ \circ (\theta \times \theta) 
: |R_p(A)|^+ \times |R_q(B)|^+ \to BGL(A \grad B)^+
\]
are homotopic. By Theorem \ref{resph} the map $\theta : |R(A \grad B)|^+ \to
BGL(A \grad B)^+$ respects the $H$-group structures up to homotopy so the maps 
\[ 
\theta \circ \gamma^{\T{rel}}_{p,q} 
\, \T{ and } \, 
\iota^+ \circ \gamma_{p,q} \circ (\theta \times \theta) 
: |R_p(A)|^+ \times |R_q(B)|^+ \to BGL(A \grad B)^+
\] 
are homotopic. The desired result now follows by uniqueness of the involved
constructions. 
\end{proof}

%%%%%%%%%%%%%%%%%%%%%%%%%%%%%%%%%%%%%%%%%%8<%%%%%%%%%%%%%%%%%%%%%%%%%%%%%%%%%%%%%%

%%%%%%%%%%%%%%%%%%%%%%%%%%%%%%%%%%%%%%%%%%8<%%%%%%%%%%%%%%%%%%%%%%%%%%%%%%%%%%%%%%

%%%%%%%%%%%%%%%%%%%%%%%%%%%%%%%%%%%%%%%%%%8<%%%%%%%%%%%%%%%%%%%%%%%%%%%%%%%%%%%%%%

%%%%%%%%%%%%%%%%%%%%%%%%%%%%%%%%%%%%%%%%%%8<%%%%%%%%%%%%%%%%%%%%%%%%%%%%%%%%%%%%%%

\section{On the multiplicative properties of the relative Chern
  character}\label{mulrel} 
Let $A$ be a unital Banach algebra. Let us start by recalling the construction
of the relative Chern character as introduced by A. Connes and M. Karoubi,
\cite{ConKar,Karoubi}. By definition, the relative Chern character is obtained
as the composition of four maps 
\[
\arr{ccc}{
\T{ch}^{\T{rel}} : K_n^{\T{rel}}(A) \to HC_{n-1}(A) 
& \q & \T{ch}^{\T{rel}} = \T{TR} \circ \epsi \circ L \circ h_n
}
\]
We will give a brief description of each of the maps. 

The first map is the Hurewicz homomorphism associated with the pointed
topological space $|R(A)|^+$, 
\[
h_n : K_n^{\T{rel}}(A) = \pi_n(|R(A)|^+) \to H_n(|R(A)|^+) \cong H_n(R(A))
\]

The second map is the logarithm 
\[
L : H_n(R(A)) \cong H_n(R^\infty(A)) \to
\lim_{p \to \infty} \hlie{M_p(A)}{n} 
\]
which is given by the chain map 
\begin{equation}\label{eq:logarithm}
L : \sigma \mapsto \int_{\Delta^n} 
\diff{\sigma}{t_1}\cdot \sigma^{-1} 
\wlw \diff{\sigma}{t_n}\cdot \sigma^{-1} 
dt_1 \ldots dt_n 
\end{equation} 
Here $\sigma : \Delta^n \to GL_p(A)$ is a smooth function. See
\cite{Tillmann}. The isomorphism $H_n(R(A)) \cong H_n(R^\infty(A))$ was proved
in Lemma \ref{consmooth}. Note that we are working with the continuous Lie
algebra complex $(\Lambda_*(A),\delta)$, thus in each degree we have a Banach
space, $\Lambda_n(A)$, in the appropriate quotient norm. This is needed in
order for the above integral to make sense. See also Section \ref{extlie}. 

The third map is the antisymmetrization 
\[
\epsi : \lim_{p \to \infty} \hlie{M_p(A)}{n} 
\to \lim_{p \to \infty} HC_{n-1}(M_p(A)) 
\]
which is given by the continuous map 
\[
\epsi : x_0 \wedge x_1 \wlw x_{n-1} \mapsto 
\sum_{s \in \Sigma_{n-1}} \T{sgn}(s) x_0 \otimes x_{s(1)} \olo
x_{s(n-1)} 
\]
Again, we are working with the continuous cyclic complex $(\ccyc{A}{n},b)$,
thus in each degree we have a Banach space, $\ccyc{A}{n}$, in the appropriate
quotient norm. See \cite{Loday2,LodQuill}. 

The last map is the generalized trace on continuous cyclic homology 
\[
\T{TR} : \lim_{p\to \infty} HC_{n-1}(M_p(A)) \to HC_{n-1}(A)
\]
See \cite{Loday2} for example. 

By \cite[Theorem $3.7$]{ConKar} the relative Chern character fits in the
(up to constants) commutative diagram 
\[
 \dgARROWLENGTH=0.5\dgARROWLENGTH
\begin{diagram}
\node{\q \ldots } \arrow{e,b}{i}    
     \node{K_{n+1}^{\T{top}}(A)} \arrow{e,t}{v} \arrow{s,l}{\T{ch}^{\T{top}}_{n+1}} 
          \node{K_n^{\T{rel}}(A)} \arrow{e,t}{\theta} \arrow{s,l}{\T{ch}_n^{\T{rel}}}
               \node{K_n(A)} \arrow{e,t}{i} \arrow{s,l}{D_n}
                    \node{K_n^{\T{top}}(A)} \arrow{e,t}{v} \arrow{s,l}{\T{ch}^{\T{top}}_n}
                         \node{\ldots \q}   \\
\node{\q \ldots} \arrow{e,b}{I}
     \node{HC_{n+1}(A)} \arrow{e,b}{S} 
          \node{HC_{n-1}(A)} \arrow{e,b}{B}
               \node{HH_n(A)} \arrow{e,b}{I}
                    \node{HC_n(A)} \arrow{e,b}{S}
                         \node{\ldots \q} 
\end{diagram}
\]
Here the other columns are the Dennis trace and the topological Chern
character. The bottom row is the $SBI$-sequence in continous
homology. Remark that the relative Chern character defined in this section
differs from the one given in \cite{ConKar, Karoubi} by the constant
$(-1)^n(n-1)!$ on $K_n^{\T{rel}}(A)$. This is necessary to make the map
respect the product structures. 

\subsection{The multiplicative properties of the logarithm}\label{mullog}
Let $A$ and $B$ be unital Banach algebras. In this section we will show that
the logarithm $L : H_*(R_p(A)) \to \hlie{M_p(A)}{*}$ respects the product
structures on the homology of the simplicial sets $R_p(A)$ and the Lie algebra
homology of the Banach algebras $M_p(A)$. These exterior products were
introduced in Section \ref{extshuff} and Section \ref{extlie}. 

For each $n,p \in \nn$ and each $j \in \{1,\ldots,n\}$ we define the operator 
\[
\arr{ccc}{
\Gamma_j : C^\infty(\Delta^n,GL_p(A)) \to C^\infty(\Delta^n,M_p(A)) 
& \q & \Gamma_j : \sigma \mapsto \diff{\sigma}{t_j}\cdot \sigma^{-1} 
}
\]
Furthermore we define the wedge product 
\[
\arr{ccc}{
\gamma : C^\infty(\Delta^n,GL_p(A)) \to C^\infty(\Delta^n, \Lambda_n(M_p(A)))
& \q & \gamma(\sigma)(t) = \Gamma_1(\sigma)(t) \wlw \Gamma_n(\sigma)(t)
}
\]
Our first task is then to understand the behaviour of $\gamma$ with respect
to the exterior shuffle product. This is the content of Lemma
\ref{techgamma}. However we will start by introducing some convenient
notation. 

Let us fix to smooth maps $\sigma : \Delta^n \to GL_p(A)$ and $\tau : \Delta^m
\to GL_q(B)$ and let us choose an isomorphism $\varphi : A^p \otimes_\zz B^q
\to (A \otimes_\zz B)^{pq}$ of $(A \otimes_\zz B)$-bimodules. Furthermore we
let $(\mu,\nu) \in \Sigma_{(n,m)}$ be a fixed $(n,m)$-shuffle. To ease the
exposition we will assume that $\mu(0)=0$ and that $\nu(m-1)=n+m-1$.

Let $\{A_0,A_1,\ldots,A_{2k+1}\}$ denote the unique partition of
$\{0,1,\ldots,n+m-1\}$ satisfying the conditions  
\[
\arr{ccc}{
\bigcup_{i=0}^k A_{2i} = \T{Im}(\mu) \, , \q 
\bigcup_{i=0}^k A_{2i+1} = \T{Im}(\nu) \sep{and}
i < j \Rightarrow 
(x < y \q \forall  x\in A_i \, , \, y \in A_j)
}
\]
Let $k_i$ denote the smallest element in $A_i$ and let $k_{2k+2}=n+m$. We
associate the composition of degeneracies $s_{A_i}= s_{k_{i+1} - 1} \ldots
s_{k_i}$ to each set $A_i$ in the partition. We then have the equality 
\[
s_\nu(\sigma) \govar s_\mu(\tau) 
= s_{A_{2k+1}}\ldots s_{A_3}s_{A_1}(\sigma)
\govar s_{A_{2k}} \ldots s_{A_2}s_{A_0} (\tau) 
: \Delta^{n+m} \to GL_{pq}(A \grad B)  
\]

For each $j \in \{0,\ldots,k\}$ we let  
\[
\arr{ccc}{
E_j = \sum_{i=0}^j |A_{2i}| = \sum_{i=0}^j (k_{2i+1}-k_{2i})
\sep{and}
O_j = \sum_{i=0}^j |A_{2i+1}| = \sum_{i=0}^j (k_{2i+2}-k_{2i+1}) 
}
\]
We then define the smooth maps   
\[
\arr{ccc}{
\omega_{2j} : \Delta^{n+m} \to \clie{M_{pq}(A \grad B)}{|A_{2j}|} 
\sep{and}
\omega_{2j+1} : \Delta^{n+m} \to \clie{M_{pq}(A \grad B)}{|A_{2j+1}|} 
}
\]
by the wedge products 
\[
\begin{split}
\omega_{2j} &= s_\nu(\Gamma_{E_{j-1}+1}(\sigma \govar 1_q) 
\wlw \Gamma_{E_j}(\sigma  \govar 1_q ) ) \q \q \T{and} \\
\omega_{2j+1} &= s_\mu(\Gamma_{O_{j-1}+1}(1_p \govar \tau)
\wlw \Gamma_{O_j}(1_p \govar \tau) )
\end{split}
\]

Finally, let us recall the relations between the degeneracies and the partial
differential operators, 
\begin{equation}\label{eq:degdiff}
\arr{ccc}{
\diff{}{t_j}\circ s_i = 
\fork{ccc}{ 
s_i \circ \diff{}{t_{j-1}}  \sep{for} j>i>0 \\
s_i \circ \diff{}{t_j} \sep{for} j\leq i  
} & \q &
\diff{}{t_j} \circ s_0 = 
\fork{ccc}{
s_0 \circ \diff{}{t_{j-1}} \sep{for} j > 1 \\
0 \sep{for} j = 1
}
}
\end{equation}

We can then prove the following technical result, 

\begin{lemma}\label{techgamma} 
Let $\sigma : \Delta^n \to GL_p(A)$ and $\tau : \Delta^m \to GL_q(B)$ be a
pair of smooth maps. For each $(n,m)$-shuffle $(\mu,\nu) \in \Sigma_{(n,m)}$
we then have the equality 
\[
\gamma(s_\nu(\sigma) \govar s_\mu(\tau)) 
= \E{sgn}(\mu,\nu) 
s_\nu(\gamma(\sigma \govar 1_q ) )
\wedge s_\mu(\gamma(1_p \govar \tau))
\]
between smooth maps $\Delta^{n+m} \to \clie{M_{pq}(A \grad B)}{n+m}$
\end{lemma}
\begin{proof}
We will assume that $\mu(0)=0$ and that $\nu(m-1)=n+m-1$. The other cases can
be proved using similar arguments. 

We start by noting that 
\[
\begin{split}
\omega_0 \wedge \omega_1 \wlw \omega_{2k+1} 
& = \T{sgn}(\mu,\nu) (\omega_0 \wedge \omega_2 \wlw \omega_{2k})
\wedge (\omega_1 \wedge \omega_3 \wlw \omega_{2k+1}) \\
& = \T{sgn}(\mu,\nu) s_\nu(\gamma(\sigma \govar 1_q)) 
\wedge s_\mu(\gamma(1_p \govar \tau)) 
\end{split}
\]
It is therefore sufficient to prove the identity
\[
\gamma(s_\nu(\sigma) \govar s_\mu(\tau))
= \omega_0 \wlw \omega_{2k+1} 
\]
We will use induction to show that   
\[
\Gamma_1(s_\nu(\sigma) \govar s_\mu(\tau)) 
\wlw \Gamma_{k_i}(s_\nu(\sigma) \govar s_\mu(\tau))  
= \omega_0 \wlw \omega_{i-1}
\]
for each $i \in \{1,\ldots,2k+2\}$. Thus, let $j \in \{1,\ldots, k_1\}$. By
the identities in \eqref{eq:degdiff} we get 
\[
\begin{split}
\diff{}{t_j}(s_\nu(\sigma)\govar s_\mu(\tau)) 
& = \diff{}{t_j}(s_\nu(\sigma)) \govar s_\mu(\tau) 
+ s_\nu(\sigma) \govar \diff{}{t_j}(s_\mu(\tau)) \\
& = s_\nu(\diff{\sigma}{t_j})\govar s_\mu(\tau)
\end{split}
\]
By consequence we have that 
\[
\Gamma_j(s_\nu(\sigma)\govar s_\mu(\tau)) 
= s_\nu(\Gamma_j(\sigma)) \govar 1_q 
= s_\nu(\Gamma_j(\sigma \govar 1_q))
\]
proving the induction start. 

Now, suppose that 
\[
\Gamma_1(s_\nu(\sigma) \govar s_\mu(\tau)) \wlw 
\Gamma_{k_i}(s_\nu(\sigma) \govar s_\mu(\tau))  
= \omega_0 \wlw \omega_{i-1}
\]
for some $i \in \{1,\ldots,2k+1\}$. We will only consider the case of $i=2r$
being even. The odd case can be proven by similar arguments. Thus, let $j \in
\{k_{2r}+1,\ldots,k_{2r+1}\}$. By the identities in \eqref{eq:degdiff} we get 
\[
\begin{split}
\diff{}{t_j}(s_\nu(\sigma)\govar s_\mu(\tau)) 
& = s_{A_{2k+1}} \ldots s_{A_{2r+1}} \diff{}{t_j}(s_{A_{2r-1}} \ldots
s_{A_1}(\sigma)) \govar s_\mu(\tau) \\
& \q + s_\nu(\sigma) \govar s_{A_{2k}} \ldots s_{A_{2r+2}} \diff{}{t_j}(s_{A_{2r}}
\ldots s_{A_0}(\tau)) \\
& = s_\nu(\diff{\sigma}{t_{j-O_{r-1}}}) \govar s_\mu(\tau)
+ s_\nu(\sigma) \govar s_\mu(\diff{\tau}{t_{k_{2r}-E_{r-1}}})
\end{split}
\]
Noting that $k_{2r}-E_{r-1} = O_{r-1}$ we deduce the identity 
\[
\Gamma_j(s_\nu(\sigma)\govar s_\mu(\tau))
= s_\nu(\Gamma_{j-O_{r-1}}(\sigma \govar 1_q)) 
+ s_\mu(\Gamma_{O_{r-1}}(1_p \govar \tau))
\]
But the term $s_\mu(\Gamma_{O_{r-1}}(1_p \govar \tau))$ already appears in the
wedge product 
\[
\omega_{2r-1}= s_\mu(\Gamma_{O_{r-2}+1}(1_p  \govar \tau) 
\wlw \Gamma_{O_{r-1}}(1_p  \govar \tau )) 
\]
Using the induction hypothesis we thus get that
\[
\begin{split}
& \Gamma_1(s_\nu(\sigma) \govar s_\mu(\tau)) \wlw 
\Gamma_{k_{2r+1}}(s_\nu(\sigma) \govar s_\mu(\tau)) \\
& \q = \big(\omega_0 \wlw \omega_{2r-1}\big)
\wedge s_\nu\big(\Gamma_{k_{2r}+1-O_{r-1}}(\sigma \govar 1_q)
\wlw \Gamma_{k_{2r+1}-O_{r-1}}(\sigma \govar 1_q)\big) \\
& \q = \omega_0 \wlw \omega_{2r}
\end{split}
\]
proving the induction step. 
\end{proof}

To continue further, we will need the following Lemma which can be proved by a
direct but tedious computation, 

\begin{lemma}\label{intsimp}
For any pair of continuous maps $\alpha : \Delta^n \to A$ and $\beta :
\Delta^m \to B$ we have the identity 
\[
\begin{split}
& \sum_{(\mu,\nu) \in \Sigma_{(n,m)}}
\int_{\Delta^{n+m}} s_\nu(\alpha) \otimes s_\mu(\beta) 
\, dt_1 \ldots dt_{n+m} \\ 
& \qq = 
(\int_{\Delta^n}\alpha\, dt_1 \ldots dt_n)
\otimes (\int_{\Delta^m}\beta\, dt_1 \ldots dt_m) 
\end{split}
\]
in the unital Banach algebra $A \grad B$. 
\end{lemma}

Let $\phi : M_p(A) \grad M_q(B) \to M_{pq}(A \grad B)$ denote the continuous
algebra homomorphism associated with the choice of the isomorphism $\varphi :
A^p \otimes_\zz B^q \to (A \otimes_\zz B)^{pq}$ of $(A \otimes_\zz
B)$-bimodules. We are now ready for the main result of this section,

\begin{prop}\label{logprod}
For each pair of smooth maps $\sigma : \Delta^n \to GL_p(A)$ and $\tau :
\Delta^m \to GL_q(B)$ we have the equality  
\[
L(\sigma \times_\varphi \tau) = \phi_*( L(\sigma) \wedge^E L(\tau) ) 
\]
in $\clie{M_{pq}(A \grad B)}{n+m}$. 
\end{prop}
\begin{proof}
Using Lemma \ref{techgamma} and Lemma \ref{intsimp} we get that 
\[
\begin{split}
L(\sigma \times_\varphi \tau) 
& = \sum_{(\mu,\nu)\in \Sigma_{(n,m)}}\T{sgn}(\mu,\nu) 
\int_{\Delta^{n+m}} 
\gamma(s_\nu(\sigma) \govar s_\mu(\tau)) 
\,dt_1\ldots dt_{n+m} \\
& = \sum_{(\mu,\nu) \in \Sigma_{(n,m)}} 
\int_{\Delta^{n+m}}
s_\nu\big(\gamma(\sigma \govar 1_q)\big) 
\wedge 
s_\mu\big(\gamma(1_p \govar \tau) \big) 
\,dt_1\ldots dt_{n+m} \\
& = L(\sigma \govar 1_q) \wedge L(1_p \govar \tau)  
\end{split}
\]
The desired result now follows by naturality of the logarithm. 
\end{proof}

\subsection{The multiplicative properties of the antisymmetrization}\label{mulant}
Let $A$ and $B$ be unital Banach algebras. We will now show that the
antisymmetrization $\epsi : \clie{A}{*} \to \ccyc{A}{*-1}$ respects the
product structures on the continuous Lie algebra homology and the
continuous cyclic homology. The definition of the
antisymmetrization is recalled in the beginning of this section and the
exterior products considered are defined in Section \ref{extlie} and Section
\ref{extcyc}. 

\begin{prop}\label{antprod} 
For each $x \in \clie{A}{n}$ and each $y \in \clie{B}{m}$ we have 
the equality 
\[
\epsi(x \wedge^E y) = \epsi(x) * \epsi(y)
\]
in $\ccyc{A \grad B}{n+m-1}$.
\end{prop}
\begin{proof}
Let $x = x_0 \wedge x_1 \wlw x_{n-1} \in \clie{A}{n}$ and let $y = y_0 \wedge
y_1 \wlw y_{m-1} \in \clie{B}{m}$. For each $i \in \{0,1,\ldots,n+m-1\}$, let 
\[
z_i = \fork{ccc}{
x_i \otimes 1_B \sep{for}  i \in \{0,\ldots,n-1\} \\
1_A \otimes y_{i-n} \sep{for} i \in \{n,\ldots,n+m-1\}
}
\]
By definition of the exterior wedge product and the antisymmetrization map we
get 
\[
\epsi(x \wedge^E y) = \sum_{s \in \Sigma_{n+m-1}}\T{sgn}(s)
z_0 \otimes z_{s(1)} \olo z_{s(n+m-1)}
\]
However, using the bijective correspondence 
\[
\arr{ccc}{
\Sigma_{(n-1,m)} \times (\Sigma_{n-1} \times \Sigma_m) \to \Sigma_{n+m-1} 
& \q & (\mu, (\sigma \times \tau)) \mapsto \mu \circ (\sigma \times \tau)
}
\]
we recognize the right hand side as the exterior Hochschild shuffle product of
the elements 
\[
\arr{ccc}{
\sum_{\sigma\in \Sigma_{n-1}}\T{sgn}(\sigma)
x_0 \otimes x_{\sigma(1)}\olo x_{\sigma(n-1)} 
\sep{and} 
\sum_{\tau \in \Sigma_m} \T{sgn}(\tau) 
1_B \otimes y_{\tau(0)} \olo y_{\tau(m-1)}  
}
\]
We therefore have  
\[
\epsi(x \wedge^E y) 
= \epsi(x) \times \epsi(1_B \wedge y) 
= \epsi(x) \times (sN\epsi)(y) = \epsi(x) * \epsi(y)
\]
proving the desired result. 
\end{proof}

\subsection{The multiplicative properties of the generalized
  trace}\label{multrace}
In this section we will show that the generalized trace $\T{TR} :
\ccyc{M_p(A)}{*} \to \ccyc{A}{*}$ respects the exterior product of degree one
in continuous cyclic homology.

Let $\phi : M_p(A) \grad M_q(B) \to M_{pq}(A \grad B)$ denote the continuous
algebra homomorphism induced by some isomorphism $\varphi : A^p \otimes_\zz B^q
\to (A \otimes_\zz B)^{pq}$ of $(A \otimes_\zz B)$-bimodules. 

\begin{prop}\label{traceprod}
For each $x \in \ccyc{M_p(A)}{n}$ and each $y \in \ccyc{M_q(B)}{m}$ we have the equality 
\[
\E{TR}(x) * \E{TR}(y) = (\E{TR}\circ \phi_*)(x * y)
\]
in $\ccyc{A \grad B}{n+m+1}$. 
\end{prop}
\begin{proof}
Let $u \in M_p(\cc)$ and let $v \in M_q(\cc)$. We start by noticing the
identity $\T{Tr}(\phi(u \otimes v)) = \T{Tr}(u)\T{Tr}(v)$. Here $\T{Tr} :
M_k(\cc) \to \cc$ denotes the usual trace.  

Using the formula of \cite[Lemma $1.2.2$]{Loday2} for the generalized trace we
thus get that
\[
(\T{TR} \circ \phi_*)(x \times (sN)(y)) 
= \T{TR}(x) \times (\T{TR}sN)(y) 
\]
The result of the Theorem then follows from the identity $(\T{TR}sN)y) =
(sN\T{TR})(y)$, see \cite[Lemma $2.2.8$]{Loday2} for example.
\end{proof}

\subsection{The multiplicative properties of the Hurewicz
  homomorphism}\label{mulhur} 
Let $A$ and $B$ be unital Banach algebras. In this section we will investigate
the behaviour of the Hurewicz homomorphism with respect to the exterior
product in relative $K$-theory and the exterior shuffle product on the
homology of the simplicial sets $R_p(A)$. The exterior product in relative
$K$-theory was constructed in Section \ref{extrel} and the exterior shuffle
product was defined in Section \ref{extshuff}. 

For each $n \in \nn$ denote the class of $1 \in \zz$ under the isomorphism
$\zz \cong H_n(S^n)$ by $\F{1}_n \in H_n(S^n)$. Furthermore, we let  
\[
\T{sh} : H_n(S^n) \otimes_\zz H_m(S^m) \to H_{n+m}(S^n \times S^m)
\]
denote the shuffle map in singular homology. Let $\pi : S^n \times S^m \to S^n
\wedge S^m$ denote the quotient map. We then get the equality
\begin{equation}\label{eq:hurprod}
(\pi_* \circ \T{sh})(\F{1}_n \otimes \F{1}_m) = \F{1}_{n+m}
\end{equation}
in $H_{n+m}(S^{n+m})\cong \zz$. For notational reasons we define 
\[ 
\zeta := \T{sh}(\F{1}_n \otimes \F{1}_m) \in H_{n+m}(S^n \times S^m)
\]
The next Lemma is the first step needed in order to express the Hurewicz
homomorphism of a product in terms of the Hurewicz homomorphism of the
original elements.

\begin{lemma}\label{hurtens}
Let $p,q \in \{3,4,\ldots\}$. Let $f : S^n \to |R_p(A)|^+ $ and $g : S^m \to
|R_q(B)|^+$ be continuous maps. We then have the equality
\[
(f \grad^+ g)_*(\zeta) = \iota_*(h_n(f) \times h_m(g))
\]
in $H_{n+m}(R(A \grad B))\cong H_{n+m}(|R(A \grad B)|^+)$. Here $\iota :
R_{pq}(A \grad B) \to R(A \grad B)$ denotes the inclusion. 
\end{lemma}
\begin{proof}
Let us fix an isomorphism $\varphi : A^p \otimes_\zz B^q \to (A \otimes_\zz
B)^{pq}$ of $(A \otimes_\zz B)$-bimodules. 

Up to canonical identifications in homology we get that the compositions 
\[
\begin{split}
& \grad^+_* \circ \T{sh} : H_n(|R_p(A)|^+) \otimes_\zz H_m(|R_q(B)|^+) \\
& \qqqq \to H_{n+m}(|R_p(A)|^+ \times |R_q(B)|^+) 
\to H_{n+m}(|R(A \grad B)|^+) \q \T{and} \\  
& \iota_* \circ (\govar)_* \circ \T{sh} : 
H_n(R_p(A)) \otimes_\zz H_m(R_q(B)) \\ 
& \qqqq \to H_{n+m}(R_p(A) \times R_q(B)) 
\to H_{n+m}(R(A \grad B)) 
\end{split}
\]
coincide. See Section \ref{extshuff} and Section \ref{conprod}. 

By definition of the exterior shuffle product and the Hurewicz homomorphism we
thus have 
\[
\iota_*( h_n(f) \times h_m(g) ) 
= (\grad^+_* \circ \T{sh})(f_*(\F{1}_n) \otimes g_*(\F{1}_m))
= (\grad^+_* \circ (f \times g)_*)(\zeta) 
\]
proving the Lemma. 
\end{proof}

The combination of the next Lemma and Lemma \ref{hurtens} entails that the
Hurewicz homomorphism respects the product structures in an appropriate sense.

\begin{lemma}\label{hurrel}
Suppose that the elements $x \in K_n^{\E{rel}}(A)$ and $y \in
K_m^{\E{rel}}(B)$ are represented by the continuous maps 
\[
\arr{ccc}{
f : S^n \to |R_p(A)|^+ \subseteq |R(A)|^+ 
\sep{and} 
g : S^m \to |R_q(B)|^+ \subseteq |R(B)|^+ 
}
\]
respectively. We then have the equality 
\[
h_{n+m}(x *^{\E{rel}} y) = (\gamma^{\E{rel}}_{p,\,q} \circ (f \times g))_*(\zeta)
\]
in $H_{n+m}(|R(A \grad B)|^+)$. Here $\gamma^{\E{rel}}_{p,\,q} : |R_p(A)|^+
\times |R_q(B)|^+ \to |R(A \grad B)|^+$ denotes the product map constructed in
Section \ref{conprod}. 
\end{lemma}
\begin{proof}
By definition, the product $x*^{\T{rel}}y \in K_{n+m}^{\T{rel}}(A \grad B)$ is
represented by the map 
\[
\widehat{\gamma}^{\T{rel}}_{p,\,q} \circ (f \wedge g) 
: S^n \wedge S^m \to |R(A \grad B)|^+
\]
Using \eqref{eq:hurprod} we thus get that the Hurewicz homomorphism of the
product is given by 
\[
\begin{split}
h_{n+m}(x*^{\T{rel}}y) 
&= (\widehat{\gamma}^{\T{rel}}_{p,\,q} \circ (f \wedge g))_*(\F{1}_{n+m}) \\
&= (\widehat{\gamma}^{\T{rel}}_{p,\,q} \circ (f \wedge g) \circ \pi)_*(\zeta) \\
&= (\widehat{\gamma}^{\T{rel}}_{p,\,q} \circ \pi \circ (f \times g))_*(\zeta) 
\end{split}
\]
The result of the lemma now follows by noting that the maps 
\[
\arr{ccc}{
\widehat{\gamma}^{\T{rel}}_{p,\,q} \circ \pi
\sep{and}  
\gamma^{\T{rel}}_{p,\,q} : |R_p(A)|^+ \times |R_q(B)|^+ \to |R(A \grad B)|^+ 
}
\]
are homotopic. See Section \ref{conprod}. 
\end{proof}

\subsection{The relative Chern character respects the exterior products}
We are now ready to prove the main result of this part of the paper: The
counterpart in continuous cyclic homology of the exterior product in relative
$K$-theory is given by the exterior product of degree one. The relevant
multiplicative structures are described in Section \ref{conprod} and Section
\ref{extcyc}.

Let $+ : HC_*(A) \oplus HC_*(A) \to HC_*(A)$ denote the addition on the
continuous cyclic homology groups. Let $\pi : H_*(R(A) \times R(A)) \to
H_*(R(A)) \oplus H_*(R(A))$ denote the map induced by the projection onto each
factor. Furthermore, let $\oplus : R(A) \times R(A) \to R(A)$ denote the
pointwise direct sum as introduced in Section \ref{hstruct}. We will need the
following preliminary result on the additive structures.

\begin{lemma}\label{sumresp}
We have the equality 
\[
+ \circ 
\big((\E{TR} \circ \epsi \circ L) \oplus (\E{TR} \circ \epsi \circ L) \big) 
\circ \pi  
= \E{TR} \circ \epsi \circ L \circ \oplus_*  
\]
between maps $H_n(R(A) \times R(A)) \to HC_{n-1}(A)$. Here we have suppressed
the identification $H_n(R(A)) \cong H_n(R^\infty(A))$ of Lemma
\ref{consmooth}.  
\end{lemma}
\begin{proof}
The result is a consequence of the naturality of the involved maps and the
behaviour of the generalized trace with respect to the direct sum operation. 
\end{proof}

Using the work accomplished in Section \ref{mullog}, \ref{mulant},
\ref{multrace} and \ref{mulhur} we are now able to prove the first main
theorem of this paper. 

\begin{prop}\label{chernprod} 
For each $x \in K_n^{\E{rel}}(A)$ and each $y \in K_m^{\E{rel}}(B)$ we have the
equality 
\[
\E{ch}^{\E{rel}}(x *^\E{rel} y) = \E{ch}^{\E{rel}}(x) * \E{ch}^{\E{rel}}(y)
\]
in $\hcyc{A \grad B}{n+m-1}$. 
\end{prop}
\begin{proof}
Suppose that $x \in K_n^{\T{rel}}(A)$ and $y \in K_m^{\T{rel}}(B)$ are
represented by the maps $f : S^n \to |R_p(A)|^+\subseteq |R(A)|^+$ and $g :
S^m \to |R_q(B)|^+ \subseteq |R(B)|^+$ respectively. By Lemma \ref{hurrel} we
have 
\[
\begin{split}
\T{ch}^{\T{rel}}(x *^\T{rel} y) 
& = (\T{TR} \circ \epsi \circ L \circ h_{n+m})(x*^{\T{rel}} y) \\
& = (\T{TR} \circ \epsi \circ L) 
\circ  (\gamma_{p, q}^\T{rel} \circ (f \times g))_*(\zeta) 
\end{split}
\]
However it follows by definition of $\gamma_{p, q}^{\T{rel}} : |R_p(A)|^+
\times |R_q(B)|^+ \to |R(A \grad B)|^+$ and by Lemma \ref{hurtens} and Lemma
\ref{sumresp} that 
\[
\begin{split}
& (\T{TR} \circ \epsi \circ L) \circ  
(\gamma_{p, q}^{\T{rel}} \circ (f \times g))_*(\zeta) \\ 
& \q = (\T{TR} \circ \epsi \circ L)(h_n(f) \times h_m(g) ) 
+ (\T{TR} \circ \epsi \circ L)\big((V_* \circ \iota_*)(h_n(1_p) \times h_m(g))\big) \\  
& \qq + (\T{TR} \circ \epsi \circ L)\big((V_*\circ \iota_*)(h_n(f) \times h_m(1_q))\big) 
+ (\T{TR} \circ \epsi \circ L)(h_n(1_p) \times h_m(1_q))
\end{split}
\]
Here $V : |R(A \grad B)|^+ \to |R(A \grad B)|^+$ denotes the homotopy inverse
of the $H$-group structure on $|R(A \grad B)|^+$, see Section
\ref{hstruct}. But the elements $h_n(1_p) \in H_n(R_p(A))$ and $h_m(1_q) \in
H_m(R_q(B))$ are both trivial so we must have 
\[
\T{ch}^{\T{rel}}(x *^\T{rel} y) 
= (\T{TR} \circ \epsi \circ L)(h_n(f) \times h_m(g) )
\]
The conclusion of the Theorem now follows from a combination of the results in
Theorem \ref{logprod}, Theorem \ref{antprod} and Theorem \ref{traceprod}. 
\end{proof}

%%%%%%%%%%%%%%%%%%%%%%%%%%%%%%%%%%%%%%%%%%8<%%%%%%%%%%%%%%%%%%%%%%%%%%%%%%%%%%%%%%

%%%%%%%%%%%%%%%%%%%%%%%%%%%%%%%%%%%%%%%%%%8<%%%%%%%%%%%%%%%%%%%%%%%%%%%%%%%%%%%%%%

%%%%%%%%%%%%%%%%%%%%%%%%%%%%%%%%%%%%%%%%%%8<%%%%%%%%%%%%%%%%%%%%%%%%%%%%%%%%%%%%%%

%%%%%%%%%%%%%%%%%%%%%%%%%%%%%%%%%%%%%%%%%%8<%%%%%%%%%%%%%%%%%%%%%%%%%%%%%%%%%%%%%%

\section{A calculation of the multiplicative character}\label{mulcombcalc} 
We start this section by briefly recalling the construction of the
multiplicative character as given in \cite{ConKar}. See also \cite{Kaminker,
  Rosenberg2}.

Let $(F,H)$ denote an odd $2p$-summable Fredholm module over a unital Banach
algebra $A$. To ease the exposition we will assume that the representation
$\pi : A \to \C L(H)$ and the map $a \mapsto [F,\pi(a)] \in \C L^{2p}(H)$ are
both continuous. We will always suppress the representation. Remark that the
conditions on continuity are not necessary for the construction of the
multiplicative character to work. They are however convenient for our
exposition. See \cite{ConKar}.

The continuous linear map 
\[
\arr{ccc}{
\tau_{2p-1} : \ccyc{A}{2p-1} \to \cc 
& \q & 
(a_0,\ldots,a_{2p-1}) \mapsto 
\frac{c_p}{2^{2p}} \T{Tr}(F[F,a_0]\clc [F,a_{2p-1}])
}
\]
where $c_p = (-1)^{p-1}\frac{(2p-1)!}{(p-1)!}$ then determines a continuous
cyclic cocycle and consequently a homomorphism 
\[
\tau_{2p-1} : HC_{2p-1}(A) \to \cc
\]
See \cite{Connes, Connesgeom}. The composition of this index cocycle with the
relative Chern character thus yields a homomorphism 
\[
\tau_{2p-1} \circ \T{ch}_{2p}^{\T{rel}} : K_{2p}^{\T{rel}}(A) \to \cc
\]
This is the additive character of the Fredholm module. 

The next step in the construction then consists of showing that the image of
the composition 
\[
\tau_{2p-1} \circ \T{ch}_{2p}^{\T{rel}} \circ v : K_{2p+1}^{\T{top}}(A) \to \cc
\]
is contained in the additive subgroup $(2\pi i)^p \zz \subseteq \cc$. Here the
map $v : K_{2p+1}^{\T{top}}(A) \to K_{2p}^{\T{rel}}(A)$ is the boundary map of
the long exact sequence \eqref{eq:xseqk} in Section \ref{extrel}. This is
accomplished in \cite[Section $4.10$]{ConKar}. By consequence the additive
character descends to a homomorphism 
\[
M_F : \T{Coker}(v) \cong \T{Im}(\theta) \to \cc/(2\pi i)^p \zz
\]
This is the odd multiplicative character associated with the odd $2p$-summable
Fredholm module $(F,H)$. With some further effort the multiplicative character
can be extended to a map on algebraic $K$-theory, however we will only need
the restriction to the subgroup $\T{Im}(\theta) \subseteq K_{2p}(A)$ for our
calculations. Again, remark that the multiplicative character defined in this
section differs from the one given in \cite{ConKar} by the constant $(2p-1)!$
on $K_{2p}(A)$. 

\subsection{The relative Chern character of a product of
  contractions}\label{cherncontract} 
Let $A$ be a \emph{commutative}, unital Banach algebra. In this section we
will give a concrete formula for the application of the relative Chern
character to products of certain elements in relative $K$-theory. We will make
use of the multiplicative properties of the relative Chern character which we
investigated in Section \ref{mulrel}.

We let
\[
\arr{ccc}{ 
*^{\T{rel}} : \relkt{A}{n} \times \relkt{A}{m} \to \relkt{A}{n+m} 
\sep{and} 
* : \hcyc{A}{n-1} \otimes_\cc \hcyc{A}{m-1} \to \hcyc{A}{n+m-1}
}
\]
denote the (interior) products in relative $K$-theory and continuous cyclic
homology. Note that these products are only available by the commutativity
assumption on $A$. See Section \ref{conprod} and Section \ref{extcyc}. 

Furthermore, for each $a \in M_\infty(A)$ we let $\gamma_a \in R(A)_1$ denote
the smooth path defined by 
\[
\arr{ccc}{
\gamma_a(t)=e^{-ta} \sep{for all} t \in [0,1]
}
\]

\begin{prop}\label{cheprod}
Let $a_0,\ldots,a_{2p-1} \in M_\infty(A)$. The relative Chern character of the
product 
\[
x = [\gamma_{a_0}]*^{\E{rel}} \ldots *^{\E{rel}} [\gamma_{a_{2p-1}}] 
\in K_{2p}^{\E{rel}}(A)
\]
is given by 
\[
\E{ch}^{\E{rel}}(x) 
= \sum_{\mu \in \Sigma_{2p-1}}\E{sgn}(\mu)
\E{TR}(a_0) \otimes \E{TR}(a_{\mu(1)}) \olo \E{TR}(a_{\mu(2p-1)}) 
\in \hcyc{A}{2p-1}
\]
\end{prop}
\begin{proof}
By Theorem \ref{chernprod} we have 
\[
\T{ch}^{\T{rel}}(x) 
= \T{ch}^{\T{rel}}[\gamma_{a_0}] * \ldots * 
\T{ch}^{\T{rel}}[\gamma_{a_{2p-1}}] \in HC_{2p-1}(A)
\]
Furthermore, the relative Chern character of the individual terms is given by 
\[
\T{ch}^{\T{rel}}([\gamma_a]) 
= (\T{TR} \circ \epsi \circ L)([\gamma_a]) 
= \T{TR}(\int_0^1 
\frac{d\gamma_a}{dt} \cdot \gamma_a^{-1}dt)
= -\T{TR}(a)  
\]
for each $a \in M_\infty(A)$. The desired result now follows by definition of
the product of degree one in continuous cyclic homology. 
\end{proof}

\subsection{Products of commutators and the cyclic cocycle of A. Connes}
Our purpose is now to obtain a different expression for the evaluation of the
index cocycle to certain elements in the cyclic homology group. We refer to
the beginning of Section \ref{mulcombcalc} for the definition of the index
cocycle which was (of course) originally introduced by A. Connes,
\cite{Connes, Connesgeom}. 

Let $(F,H)$ be an odd $2p$-summable Fredholm module over a commutative, unital
Banach algebra $A$. We will suppose that the representation $\pi : A \to \C
L(H)$ and the linear map $a \mapsto [F,\pi(a)] \in \C L^{2p}(H)$ are
continuous. Let $P = \frac{F+1}{2}$ denote the associated projection. 

For each $n \in \nn$, let $SE_n \subseteq \Sigma_n$ denote the subset of
permutations defined by 
\[
s \in SE_n \Leftrightarrow 
\big( s \in \Sigma_n \, \, \, \T{and} \, \, \, 
s(2i) < s(2i+1)  \big) 
\]

We then have the following combinatorial result, which can be proved by
induction. 

\begin{lemma}\label{techcomm}
For each algebra $B$ over $\cc$ we have the identity
\[
\sum_{\mu\in \Sigma_{2p}} \E{sgn}(\mu)x_{\mu(0)} \olo x_{\mu(2p-1)}
= \sum_{s \in SE_{2p}} \E{sgn}(s)X_{s(0),\,s(1)} \olo X_{s(2p-2),\, s(2p-1)} 
\]
in the tensor product $B^{\otimes 2p}$. Here $X_{s(2i),\,s(2i+1)} \in B
\otimes B$ denotes the commutator  
\[
X_{s(2i),\,s(2i+1)}= 
x_{s(2i)} \otimes x_{s(2i+1)} - x_{s(2i+1)} \otimes x_{s(2i)} 
\]
\end{lemma}

Now, let $T \in C^{2p-1}(A)$ denote the cochain given by 
\[
T : (x_0,\ldots,x_{2p-1}) 
\mapsto c_p\T{Tr}(Px_0(1-P)x_1P \clc Px_{2p-2}(1-P)x_{2p-1}P)  
\]
Here $c_p = (-1)^{p-1}\frac{(2p-1)!}{(p-1)!}$. We then note that $T \circ
t^2 = T$ and that $T \circ (1+t) = \tau_{2p-1}$, where $\tau_{2p-1} \in
\ccocyc{A}{2p-1}$ is the index cocycle associated with the odd $2p$-summable
Fredholm module $(F,H)$ over $A$. We can thus conclude that $T \circ N =
p\cdot \tau_{2p-1}$.

\begin{prop}\label{prcomm}
For each $x_0,\ldots,x_{2p-1} \in A$ we have the identity 
\[
\begin{split}
& \sum_{\mu \in \Sigma_{2p-1}}\E{sgn}(\mu) 
\tau_{2p-1}(x_0 \otimes x_{\mu(1)} \olo x_{\mu(2p-1)}) \\
& \qquad = (-1)^p c_p \sum_{s \in SE_{2p-1}}\E{sgn}(s)
\E{Tr}([Px_0P,Px_{s(1)}P] \clc  [Px_{s(2p-2)}P, Px_{s(2p-1)}P])
\end{split}
\]
\end{prop}
\begin{proof}
Using Lemma \ref{techcomm} and the considerations preceeding the statement of
the Theorem we get that
\[
\begin{split}
& \sum_{\mu \in \Sigma_{2p-1}}\T{sgn}(\mu)
\tau_{2p-1}(x_0 \otimes x_{\mu(1)} \olo x_{\mu(2p-1)}) \\
& \q = \frac{1}{p} \sum_{\mu \in \Sigma_{2p}}\T{sgn}(\mu)
T(x_{\mu(0)}\otimes x_{\mu(1)} \olo x_{\mu(2p-1)}) \\
& \q = \frac{1}{p} \sum_{s \in SE_{2p}}\T{sgn}(s)
T(X_{s(0),\,s(1)} \olo X_{s(2p-2),\,s(2p-1)})
\end{split}
\]
However, by commutativity of $A$ we get the relation 
\[
\arr{ccc}{
[PxP,PyP] = -Px(1-P)yP + Py(1-P)xP & \q & \forall x,y\in A
}
\]
It follows that 
\[
\begin{split}
& \sum_{\mu \in \Sigma_{2p-1}}\T{sgn}(\mu)
\tau_{2p-1}(x_0 \otimes x_{\mu(1)} \olo x_{\mu(2p-1)}) \\
& \q = \frac{ (-1)^p }{p} c_p \sum_{s \in SE_{2p}}\T{sgn}(s) \\
& \qq \T{Tr}([Px_{s(0)}P,Px_{s(1)}P]\clc [Px_{s(2p-2)}P,Px_{s(2p-1)}P]) \\
& \q = (-1)^p c_p \sum_{s \in SE_{2p-1}}\T{sgn}(s) \\
& \qq \T{Tr}([Px_0P,Px_{s(1)}P][Px_{s(2)}P,Px_{s(3)}P] \clc
[Px_{s(2p-2)}P,Px_{s(2p-1)}P]) 
\end{split}
\]
In the last equation we have used that the set of permutations $SE_{2p-1}$ can
be viewed as the quotient of $SE_{2p}$ by an action of the cyclic group on $p$
elements. We have thus obtained the desired result. 
\end{proof}

\subsection{An evaluation of the multiplicative character on higher Loday
  symbols}
We are now ready to prove our concrete formula for the application of the
multiplicative character to higher Loday products. This will accomplish the
main purpose of the paper. 

Let $(F,H)$ be an odd $2p$-summable Fredholm module over a commutative, unital
Banach algebra $A$. We will suppose that the representation $\pi : A \to \C
L(H)$ and the linear map $a \mapsto [F,\pi(a)] \in \C L^{2p}(H)$ are
continuous. Let $P = \frac{F+1}{2}$ denote the associated projection. We refer
to the beginning of Section \ref{mulcombcalc} for a brief reminder on the
construction of the multiplicative character. 

\begin{prop}
Let $a_0,\ldots,a_{2p-1} \in M_\infty(A)$. The multiplicative character of the
Loday product $[e^{a_0}] * \ldots * [e^{a_{2p-1}}] \in K_{2p}(A)$ is then
given by 
\[
\begin{split}
& M_F([e^{a_0}] * \ldots * [e^{a_{2p-1}}]) \\
& \q = (-1)^p c_p \sum_{s \in SE_{2p-1}}\E{sgn}(s) \\
& \qq \E{Tr}([P\E{TR}(a_0)P,P\E{TR}(a_{s(1)})P] 
\clc  [P\E{TR}(a_{s(2p-2)})P, P\E{TR}(a_{s(2p-1)})P]) \\ 
& \qqq \in \cc/(2\pi i)^p \zz
\end{split}
\]
\end{prop}
\begin{proof}
For each $i \in \{0,\ldots,2p-1\}$ we let $\gamma_{a_i} \in R(A)_1$ denote the
smooth path given by $\gamma_{a_i} : t \mapsto e^{-ta_i}$. We then have 
\[
\theta([\gamma_{a_i}]) = [\gamma_{a_i}(1)^{-1}] = [e^{a_i}]
\]
By Theorem \ref{thetprod} the map $\theta : \oplus_{n=1}^\infty
K_n^{\T{rel}}(A) \to \oplus_{n=1}^\infty K_n(A)$ is a homomorphism of graded
rings, so we get that 
\[
\theta([\gamma_{a_0}]*^{\T{rel}} \ldots *^{\T{rel}} [\gamma_{a_{2p-1}}]) 
= [e^{a_0}] * \ldots * [e^{a_{2p-1}}]
\] 
By definition of the multiplicative character we then have 
\[
M_F ( [e^{a_0}] * \ldots * [e^{a_{2p-1}}] ) 
= (\tau_{2p-1} \circ \T{ch}^{\T{rel}})
([\gamma_{a_0}]*^{\T{rel}} \ldots *^{\T{rel}} [\gamma_{a_{2p-1}}]) \in \cc/(2\pi
i)^p \zz
\]
But it follows from Theorem \ref{cheprod} and Theorem \ref{prcomm} that the
right hand side is given by 
\[
\begin{split}
& (\tau_{2p-1} \circ \T{ch}^{\T{rel}})
([\gamma_{a_0}]*^{\T{rel}} \ldots *^{\T{rel}} [\gamma_{a_{2p-1}}]) \\
& \q = \sum_{\mu \in \Sigma_{2p-1}}\T{sgn}(\mu) 
\tau_{2p-1}(\T{TR}(a_0) \otimes \T{TR}(a_{\mu(1)}) \olo \T{TR}(a_{\mu(2p-1)}))
\\ 
& \q = (-1)^p c_p \sum_{s \in SE_{2p-1}}\T{sgn}(s) \\
& \qq \T{Tr}([P\T{TR}(a_0)P,P\T{TR}(a_{s(1)})P] 
\clc  [P\T{TR}(a_{s(2p-2)})P, P\T{TR}(a_{s(2p-1)})P]) 
\end{split}
\]
proving the desired result. 
\end{proof}

\begin{corollary}
For any commutative unital Banach algebra the multiplicative character is
calculizable on the subgroup of $K_{2p}(A)$ generated by Loday products of
elements in the connected component of the identity, $GL_0(A)$.
\end{corollary}
\begin{proof}
Since each element in $GL_0(A)$ can be obtained as a product of exponentials,
the result follows by noting that the Loday product is multilinear and that
the multiplicative character is a homomorphism of abelian groups. 
\end{proof}

\end{document}